\documentclass[12pt,twoside]{amsart}
\usepackage{amssymb,amsmath,amsfonts,amsthm}
\usepackage{mathrsfs}
\usepackage[X2,T1]{fontenc}
\usepackage[applemac]{inputenc}
\usepackage{cite}
\usepackage{latexsym}
\usepackage[dvips]{graphicx}
\usepackage{comment}
\usepackage{ulem}
\def\Im{\mathop{\rm Im}\nolimits}

\def\Im{\mathop{\rm Im}\nolimits}

\def\R{\mathbb R}
\def\C{\mathbb C}
\def\N{\mathbb N}

\def\ds{\displaystyle}

\newcommand\dslash{d\llap {\raisebox{.9ex}{$\scriptstyle-\!$}}}
\usepackage{colortbl}

\newcommand{\beqsn}{\arraycolsep1.5pt\begin{eqnarray*}}
	\newcommand{\eeqsn}{\end{eqnarray*}\arraycolsep5pt}
\newcommand{\beqs}{\arraycolsep1.5pt\begin{eqnarray}}
	\newcommand{\eeqs}{\end{eqnarray}\arraycolsep5pt}

\newtheorem{Th}{Theorem}[section]
\newtheorem{Rem}[Th]{Remark}

\newtheorem{Lemma}[Th]{Lemma}
\newtheorem{Def}[Th]{Definition}
\newtheorem{Prop}[Th]{Proposition}

\catcode`\@=11

\renewcommand{\section}%
{\setcounter{equation}{0}\@startsection {section}{1}{\z@}{-3.5ex plus -1ex
		minus -.2ex}{2.3ex plus .2ex}{\Large\bf}}

\title[Schr\"odinger type equations with distributional coefficients]{Schwartz very weak solutions for Schr\"odinger type equations with distributional coefficients}
\author[Arias Junior]{Alexandre Arias Junior}
\address{Alexandre Arias Junior\\
	Department of Computer Science and Mathematics (FFCLRP), University of S\~ao Paulo (USP),
	Ribeir\~ao Preto, SP, 14040-901, Brazil
	}
	\email{alexandre.ariasjunior@usp.br}
\author[Ascanelli]{Alessia Ascanelli}
\address{Alessia Ascanelli\\
	Dipartimento di Matematica ed Informatica\\Universit\`a di Ferrara\\
	Via Machiavelli 30\\
	44121 Ferrara\\
	Italy}
\email{alessia.ascanelli@unife.it}
\author[Cappiello]{Marco Cappiello}
\address{Marco Cappiello\\
	Dipartimento di Matematica ``G. Peano" \\Universit\`a di Torino\\
	Via Carlo Alberto 10\\
	10123 Torino\\
	Italy}
\email{marco.cappiello@unito.it}
\author[Garetto]{Claudia Garetto}
\address{Claudia Garetto\\
	School of Mathematical Sciences\\Queen Mary University of London\\
Mile End Road, London E1 4NS\\ UK}
\email{c.garetto@qmul.ac.uk}


\begin{document}

\begin{abstract} 
	This paper continues the analysis of Schr\"odinger type equations with distributional coefficients initiated by the authors in \cite{AACG}. Here we consider coefficients that are tempered distributions with respect to the space variable and are continuous in time. We prove that the corresponding Cauchy problem, which in general cannot even be stated in the standard distributional setting, admits a Schwartz very weak solution which is unique modulo negligible perturbations. Consistency with the classical theory is proved in the case of regular coefficients and Schwartz Cauchy data.
\end{abstract}

\maketitle

\section{Introduction}
\noindent

Schr\"odinger equation is the fundamental equation in quantum mechanics, since it describes the evolution in time of the  wave function's state $u(t,x)$ of a quantic particle with mass $m$:
\beqs\label{sc}\hbar \partial_t u= -\ds\frac{\hbar^2}{2m}\triangle u+Vu, \ t\in[0,T],\ x\in\R^n,\eeqs
where $\hbar$ is the reduced Planck constant, V is a potential and, as usual $\triangle:=\ds\sum_{j=1}^n \partial_{x_j}^2$ is the Laplace operator. Literature on the topic is really wide, in particular in the case of a power potential $Vu=\lambda u|u|^{2k}$, $k>0,$ $\lambda \in \R$, and various questions of existence of a unique solution in Sobolev spaces have been studied, see for instance \cite{Kenig_SG} and the references therein. Also in the case of space and time depending coefficients the literature is extensive, not only for the Schr\"odinger equation itself but also for the so-called class of Schr\"odinger type equations with lower order terms $Su=f$, where 
\begin{equation}\label{city_of_tears}
S = D_t + \sum_{j=1}^{n}D^2_{x_j} + \sum_{j=1}^{n}c_j(t,x)D_{x_j} + c_0(t,x), \quad D=-i\partial,
\end{equation}
with at least an assumption of continuity with respect to time and an assumption of smoothness in space for the coefficients. 

It is indeed well known, see \cite{ichinose_remarks_cauchy_problem_schrodinger_necessary_condition}, that
 in the case $c_j(t,x) = c_j(x) \in \mathcal{B}^\infty(\R^n)$ the Cauchy problem associated to operator $S$ may be well-posed in $H^\infty(\R^n)= \bigcap_{m \in \R} H^m(\R^n)$, only if there exist  $M,N >0$ such that
	$$
	\sup_{x \in \R^{n}, \omega \in S^{n-1}} \left| \sum_{j=1}^{n} \int_{0}^{\rho} \Im\, c_j(x+\omega\theta) d\theta \right| \leq M\log(1+\rho) +N, \quad \forall\, \rho \geq 0.
	$$
	Here $\mathcal B^\infty(\R^n)$ denotes the space of all smooth functions on $\R^n$ with uniformly bounded derivatives.
Moreover, see \cite{KB}, if 
\beqs\label{asskb}
|\Im\, c_j(t,x))|\leq C\langle x\rangle^{-\sigma}\,  \forall x\in\R^n,\ t \in [0,T], 
\eeqs
where $\langle x\rangle:=(1+|x|^2)^{1/2}$ and $\sigma \geq 1$, then the Cauchy problem associated to $S$ is well-posed in $H^m(\R^n)$ for every $m \in \R$ when $\sigma > 1$, in $H^\infty(\R^n)$ when $\sigma = 1$. In the latter case, a finite loss of derivatives appears in the solution.
 Results of well-posedness in Gevrey classes and Gelfand-Shilov spaces for operator $S$ can be found respectively in \cite{KB, CicRei, ACR} and \cite{Arias_GS, scncpp5}. A necessary condition for well posedness in a Gevrey class of index $\theta>1$ can be found in \cite{AAMnec}.

We point out that a result of well-posedness in the Schwartz space of rapidly decreasing functions $\mathscr S(\R^n)$ for the Cauchy problem associated to the operator $S$, which is interesting per se, 
appears as a byproduct in the present paper, see Theorem \ref{Th_classical_Schwartz_result} at Section \ref{section_Schwartz_classical_result}.

\medskip

This paper concerns the following question: what happens when the coefficients are irregular, that is if we lose the minimal regularity assumptions on the coefficients and they turn out to be discontinuous functions, or distributions? We focus on the Cauchy problem 
\begin{equation}\label{CPintro}
	\begin{cases}
		S u(t,x) = f(t,x), \quad t \in [0,T],\, x \in \R^{n}, \\
		u(0,x) = g(x), \quad \quad \,\, x \in \R^{n},
	\end{cases}
\end{equation}
for the Schr\"odinger-type operator given in \eqref{city_of_tears} 
with the following assumption on the lower order coefficients
\begin{equation}\label{Alexander_the_Great}
	c_j(t,x) \in C([0,T]; \mathscr{S}'(\R^n)), \quad j = 0, 1, \ldots, n,
\end{equation}
and we take inital data satisfying
\begin{equation}\label{revelations}
	f(t,x) \in C([0,T]; \mathscr{S}'(\R^n)), \quad g \in \mathscr{S}'(\R^n),
\end{equation}
where, as usual, $\mathscr{S}'(\R^n)$ stands for the space of tempered distributions.

In this case the Cauchy problem \eqref{CPintro} is ill-posed either in Sobolev-type spaces or in Gevrey type spaces, cf. \cite{ichinose_remarks_cauchy_problem_schrodinger_necessary_condition, AAMnec}. Moreover, due to the remarkable Schwartz impossibility result, 
the operator $S$ might fail to be defined as an operator acting in the space $C([0,T]; \mathscr{S}'(\R^n))$. Likewise, in this general setting, a non-trivial matter is to define what a solution to \eqref{CPintro} should be. 

\medskip
The concept we are going to deal with is the notion of Schwartz very weak solution, see \cite{AACG,G,GR}. Very weak solutions have been introduced in \cite{GR} to provide a meaningful notion of solution for hyperbolic equations with highly irregular coefficients, namely distributions. See also \cite{CardonaRuzhansky, DGL, DGL2, G, GS24, MRT19}. The main idea is to replace the original equation with a family of regularised equations depending on the parameter $\epsilon$ and to investigate the $\epsilon$-behaviour of the net of corresponding solutions. While in \cite{GR} very weak solutions were modelled on Gevrey classes, here for the analysis of the Schr\"odinger type equations we are interested in, we will work with the space $\mathscr{S}(\R^n)$ of smooth and rapidly decreasing functions (Schwartz functions).

To construct a $\mathscr{S}$-very weak solution to \eqref{CPintro}, we take initial data $f \in C([0,T]; \mathscr{S}'(\R^{n}))$ and $g \in \mathscr{S}'(\R^{n})$, and we follow an approach similar to the one introduced in \cite{GR}  as follows:
\begin{enumerate}
\item we regularise the tempered distributions $c_j(t),f(t),g$ by a special net of Schwartz mollifiers parametrised by a	 positive scale $\omega(\varepsilon)$ converging to $0$ as $\epsilon\to 0^+$; this regula\-ri\-sa\-tion produces a family of Schwartz functions indexed by $\epsilon>0$ and converging in $\mathscr S'(\R^n)$ to the original tempered distributions as $\epsilon\to 0^+$, see the Appendix  (Section \ref{Regularisation of tempered distributions by Schwartz functions}) for the details;
\item we then obtain a family of regularised Cauchy problems 
\begin{equation}\label{regularised_CPintro}
		\begin{cases}
			S_{\varepsilon} v(t,x) = f_\epsilon(t,x), \quad t \in [0,T],\, x \in \R^{n}, \\
			v(0,x) = g_\epsilon(x), \quad \quad \,\, \, x \in \R^{n},
		\end{cases}
	\end{equation}
where  
	\begin{equation}\label{regularised_opintro}
	S_{\varepsilon} = D_t + \sum_{j=1}^{n}D^{2}_{x_j} + \sum_{j=1}^{n} c_{j,\varepsilon}(t,x) D_{x_j} + c_{0,\varepsilon}(t,x).
	\end{equation}
These problems have Cauchy data and coefficients that are continuous functions in $t$ and Schwartz functions in $x$, so they fulfill assumption \eqref{asskb} and we know that the Cauchy problems \eqref{regularised_CPintro} admit unique solutions $u_\epsilon\in C([0,T],H^\infty(\R^n)).$ But we search for well-posedness in $\mathscr S$, so we prove in Theorem \ref{apprn} of Section \ref{n} that we have in fact $u_\epsilon\in C([0,T],\mathscr S(\R^n)).$ 
To conclude this we study an auxiliary Cauchy problem, given by a suitable conjugation by powers of $\langle x\rangle$, in order to conclude rapid decay for the solution $u_\varepsilon$. A careful use of microlocal techniques and a sharp use of the energy method for evolution equations is needed in the proof. In addition, we provide a qualitative analysis of the solutions $(u_\varepsilon)_\varepsilon$, by giving a suitable energy estimate.
\item Using the energy estimate, we study the behavior of the net $(u_\epsilon)_{\epsilon}$ as $\epsilon\to 0^+$, and we prove that the net $(u_\epsilon)_{\epsilon \in (0,1]}$ defines a $\mathscr S$-very weak solution to the Cauchy problem \eqref{CPintro}, that is the net is $\mathscr S$-moderate and the solution is unique modulo negligible changes, see Section \ref{Schwartz very weak solutions} for precise definitions. 
\item Finally, we show in Section \ref{Consistency with regular theory} that the result obtained is consistent with the classical theory: in the case of regular coefficients and Schwartz Cauchy data, the net converges in  $\mathscr S$ to the unique classical solution, and the limit is independent of the regularisation used. 
\end{enumerate}
\
This paper is a second step in the study of Cauchy problem \eqref{CPintro} with irregular coefficients. The first one was made in the recent paper \cite{AACG} where we dealt with a similar problem. There, we analyse coefficients of the type 
\begin{equation}\label{the_dictators}
c_j(t,x) = a_j(t)b_j(x), \quad j =0, 1, \ldots,n,
\end{equation}
where $a_j$ are distributions with compact support in $\mathcal{E}'([0,T])$ and $b_j$ belong to suitable subclasses $H^{-\infty,j}(\R^n)$ of $\mathscr S'(\R^n)$, $j\geq 2$, where $H^{-\infty,j}$ stands for the class of tempered distributions $v$ such that all the derivatives $\partial_\xi^\beta \hat v(\xi)$ have an at least polynomial growth for $|\beta|\leq j$; the standard regularisation of these distributions produces uniformly bounded functions with decay $\langle x\rangle^{-j}$, $j\geq 2$, so, by \cite{KB}, the regularised Cauchy problem turns out to be well posed in every Sobolev space and we prove that the Cauchy problem \eqref{CPintro} admits a $H^\infty$-very weak solution. 

We remark that combining the ideas of the present manuscript with the ones developed in \cite{AACG}, we might have considered coefficients of the type \eqref{the_dictators} with $a_j \in \mathcal{E}'([0,T])$ and $b_j \in \mathscr{S}'(\R^n)$ and still obtain an analogous of our main result (cf. Theorem \ref{main_theorem}). We decided to consider $c_j \in C([0,T];\mathscr{S}'(\R^n))$ in order to obtain a less technical and more elegant result. In addition, the analysis of the regularisations in the space variable is the core of the problem and the most challenging part. Note that, in order to allow tempered distributions in the variable $x$, we need to use a special kind of regularisation (cf. Appendix \ref{Regularisation of tempered distributions by Schwartz functions}) and to change the strategy for solving the regularised problem, which becomes more involved with respect to the one used in \cite{AACG}. 


The manuscript is organised as follows. In Section \ref{Schwartz very weak solutions} we introduce the notion of very weak solution of Schwartz type and we state our main result that will be proven in Sections \ref{n}, \ref{Schwartz very weak well-posedness} and \ref{Iron_maiden}. Consistency with the classical theory in the case of regular coefficients is the topic of Section \ref{Consistency with regular theory}. The final Section \ref{section_Schwartz_classical_result} concerns a well-posedness result in Schwartz space for \eqref{CPintro} in a regular functional setting. We close the paper with two appendices: Appendix  A concerns regularisation of tempered distributions by Schwartz functions and Appendix B deals with some aspects of the theory of pseudo-differential operators relevant to this paper.

	
\medskip

\smallskip


\section{Schwartz very weak solutions and main result}\label{Schwartz very weak solutions}
\

The aim of this section is to 
state the main result of this manuscript. As we already mentioned in the introduction, due to the Schwartz impossibility result about products of distributions, we need to define a suitable concept of solution. The approach that we shall adopt is the so-called very weak solutions (cf. \cite{G,GR}). 

We consider the operator \eqref{city_of_tears} under the hypothesis \eqref{Alexander_the_Great} on the coefficients. We are interested in the Cauchy problem \eqref{CPintro}, where the initial data satisfy $g \in \mathscr{S}'(\R^{n})$ and $f \in C([0,T]; \mathscr{S}'(\R^{n}))$. Once the main problem of the paper is posed, the next step is to define the notion of solution to \eqref{CPintro} that we are going to deal with. To define it, we shall need some preliminaries. \\ 
\indent
Let $\phi$ and $\psi$ be Schwartz functions such that $\phi(0) = 1 = \widehat{\psi}(0)$. Given a positive scale $\omega(\varepsilon)$, $\varepsilon \in (0,1]$, i.e. $\omega$ is positive, bounded, $\omega(\varepsilon) \to 0$ as $\varepsilon \to 0^{+}$ and $\omega(\varepsilon) \geq c \varepsilon^{r}$, for some $c, r > 0$, we define
$$
\phi^{\omega(\varepsilon)}(x) = \phi(\omega(\varepsilon)x) \quad \text{and} \quad \psi_{\omega(\varepsilon)}(x) = \frac{1}{(\omega(\varepsilon))^{n}} \psi \left( \frac{x}{\omega(\varepsilon)} \right).
$$
Now, given any $u \in \mathscr{S}'(\R^n)$, the regularisation 
$$
u_{\omega(\varepsilon)}(x) = \phi^{\omega(\varepsilon)}(x) \{ \psi_{\omega(\varepsilon)} \ast u \}(x)
$$
satisfies 
\begin{equation}\label{children_of_the_sea}
u_{\omega(\varepsilon)} \in \mathscr{S}(\R^n)\,\, \text{and} \,\, u_{\omega(\varepsilon)} \to u \,\,  \text{in} \,\, \mathscr{S}'(\R^n)\,\, \text{as} \,\, \varepsilon \to 0^{+}. 
\end{equation}
See Appendix \ref{Regularisation of tempered distributions by Schwartz functions} at end of the paper for more details on these kind of regularisations. The next proposition is standard and for this reason we are going to omit its proof.

\begin{Prop}\label{the_evil_that_men_do}\leavevmode
	\begin{itemize}
		
		\item[(i)] If $u \in \mathscr{S}(\R^{n})$ then for any $\beta \in \N_0^n$ and $M \in \N_0$ there exists $C_{\beta,M} > 0$ such that 
		$$
		|\langle x \rangle^{M} \partial^{\beta}_{x} \{\phi(\omega(\varepsilon)x)(\psi_{\omega(\varepsilon)} \ast u)(x)\}| \leq C_{\beta,M};
		$$
		
		\item[(ii)]  If $u \in \mathscr{S}(\R^{n})$, $(\partial^{\beta}\phi)(0) = 0$ for all $\beta \neq 0$ and $\int x^{\alpha}\psi(x) dx = 0$ for all $\alpha \neq 0$, then for any $\beta \in \N_0^n$ and any $q, M \in \N_0$ there exists $C > 0$ such that 
		$$
		|\langle x \rangle^{M} \partial^{\beta}_{x}\{\phi(\omega(\varepsilon) x ) (\psi_{\omega(\varepsilon)} \ast u)(x) - u(x)\}| \leq C (\omega(\varepsilon))^{q}.
		$$
	\end{itemize}
\end{Prop}

It is convenient to describe $\mathscr{S}(\R^n)$ and $\mathscr{S}'(\R^n)$ in terms of suitable weighted Sobolev spaces. For $m, M \in \R$ we define
$$
H^{m,M}(\R^n) = \{ u \in \mathscr{S}'(\R^n) : \|u\|_{H^{m,M}(\R^n)}= \|\langle x \rangle^{M} \langle D \rangle^{m} u \|_{L^2}<\infty \},
$$
where $\langle D \rangle^m$ stands for the Fourier multiplier given by the symbol $\langle \xi \rangle^{m}$. It is well-known that $H^{m,M}(\R^n)$ are Hilbert spaces and 
$$
\mathscr{S}(\R^n) = \bigcap_{m,M \in \R} H^{m,M}(\R^n), \quad \mathscr{S}'(\R^n) = \bigcup_{m,M \in \R} H^{m,M}(\R^n)
$$
with equivalent limit type topologies. For more details on weighted Sobolev spaces see Proposition $2.3$ and its corollaries in \cite{parenti_sg_calc}, cf. also \cite{Cordes}.

Proposition \ref{the_evil_that_men_do} and \eqref{children_of_the_sea} (cf. Theorem \ref{drunken_lullabies}) motivate the following Definition \ref{def_mod_negli}.

\begin{Def}\label{def_mod_negli}
	\leavevmode
	\begin{itemize}
		\item[(i)] Let $(v_{\varepsilon})_{\varepsilon}$ be a net in $\mathscr{S}(\R^{n})^{(0,1]}$. We say that the net $(v_\varepsilon)_\varepsilon$ is $\mathscr{S}$-moderate if for any $M \in \N_0$ and any $\beta \in \N_{0}^{n}$ there exists $N(M,\beta) = N \in \N_0$ and $C(M,\beta) = C > 0$ such that 
		$$
		\sup_{x\in\R^{n}} \langle x \rangle^{M} |\partial^{\beta}_{x} v_{\varepsilon}(x)| \leq C \varepsilon^{-N}, \quad \forall\, \epsilon\in (0,1].
		$$
		This is equivalent to: for all $M,m \in \N_0$ there exists $N(M,m) = N \in \N_0$ and $C(M,m) = C > 0$ such that 
		$$
		\|v_{\varepsilon}\|_{H^{m,M}(\R^{n})} \leq C \varepsilon^{-N}, \quad  \forall\, \epsilon\in (0,1].
		$$
		
		\item[(ii)] Let  $(v_{\varepsilon})_{\varepsilon}$ be a net in $\mathscr{S}(\R^{n})^{(0,1]}$. We say that the net $(v_{\varepsilon})_{\varepsilon}$ is $\mathscr{S}$-negligible if for any $M\in\N_0$, any $\beta \in \N^{n}_{0}$ and any $q\in\N_0$ there exist $C(M,\beta,q) = C > 0$ such that 
		$$
		\sup_{x \in \R^{n}}|\langle x \rangle^{M} \partial^{\beta}_{x}\partial^{\beta}_{x} v_{\varepsilon}(x)| \leq C \varepsilon^{q}, \quad \forall\, \epsilon\in (0,1].
		$$
		This is equivalent to: for any $M,m \in \N_0$ and any $q\in\N_0$  there exists $C(M,m,q) = C > 0$ such that 
		$$
		\|v_{\varepsilon}\|_{H^{m,M}(\R^{n})} \leq C \varepsilon^{q},\quad \forall\, \epsilon\in (0,1].
		$$
	\end{itemize}
\end{Def}

\begin{Rem}
	We can extend Definition \ref{def_mod_negli} to nets in $\{C([0,T];\mathscr{S}(\R^{n}))\}^{(0,1]}$ just by asking uniform estimates with respect to the variable $t$.
\end{Rem}

\begin{Rem}
	The notions of Schwartz neglibility and Schwartz moderateness also appear in \cite{G2004} (cf. Definition $2.8$) and \cite{G_Topological} (cf. Definition $3.1$).
\end{Rem}

\begin{Rem}
	Let $\mathcal{M}_{\mathscr{S}}$ be the set of all $\mathscr{S}-$moderate nets and $\mathcal{N}_{\mathscr{S}}$ be the set of all $\mathscr{S}-$negligible nets. Then define the quotient space
	$$
	\mathcal{G}_{\mathscr{S}} := \frac{\mathcal{M}_{\mathscr{S}}}{\mathcal{N}_{\mathscr{S}}}.
	$$
	The set $\mathcal{G}_{\mathscr{S}}$ is known as the Colombeau algebra associated to the locally convex space $\mathscr{S}(\R^n)$ (see Definition $3.1$ of \cite{G_Topological}).
	Next, for fixed mollifiers $\phi$ and $\psi$ such that $\phi(0) = 1 = \widehat{\psi}(0)$, $(\partial^{\beta}\phi)(0) = 0$ for all $\beta \neq 0$ and $\int x^{\alpha}\psi(x) dx = 0$ for all $\alpha \neq 0$, the map 
	$$
	\varPsi: \mathscr{S}'(\R^n) \to \mathcal{G}_{\mathscr{S}}, \quad
	\varPsi (u) = (\phi^{\varepsilon} (\psi_{\varepsilon} \ast u))_{\varepsilon \in (0,1]}, \quad \forall\,\, u \in \mathscr{S}'(\R^n),
	$$
	is a well-defined embedding (due to Theorem \ref{drunken_lullabies} and Proposition \ref{the_evil_that_men_do}). Then, the regularisation $(\phi^{\varepsilon} (\psi_{\varepsilon} \ast u))_{\varepsilon}$ gives a way to include tempered distributions in the Colombeau algebra $\mathcal{G}_{\mathscr{S}}$. Notice that in $\mathcal{G}_{\mathscr{S}}$ we have a well-defined multiplication operation which is consistent with the product of functions in $\mathscr{S}(\R^n)$.
\end{Rem}

We are finally ready to introduce the concept of very weak solution that we are interested in and the main result of this paper, namely Theorem \ref{main_theorem}. 

\begin{Def}
	\label{def_vw}
	The net $(u_\varepsilon)_{\varepsilon} \in \{C([0,T];\mathscr{S}(\R^{n}))\}^{(0,1]}$ is a $\mathscr{S}$-very weak solution for the Cauchy problem \eqref{CPintro} if $(u_\varepsilon)_\varepsilon$ is $\mathscr{S}$-moderate and there exist 
	\begin{itemize}
		
		\item $\mathscr{S}$-moderate regularisations $(c_{j,\varepsilon})_{\varepsilon}$ of the coefficients $c_j$, $j = 0, 1, \ldots, n$,
		
		\item $\mathscr{S}$-moderate regularisations $(f_{\varepsilon})_{\varepsilon}, (g_{\varepsilon})_{\varepsilon}$ of the Cauchy data $f$ and $g$,
	\end{itemize}
	such that, for every fixed $\varepsilon$, $u_{\varepsilon}$ solves the regularised Cauchy problem
	\eqref{regularised_CPintro} for the operator \eqref{regularised_opintro}
\end{Def}

\begin{Th}\label{main_theorem} Under the assumptions \eqref{Alexander_the_Great} and \eqref{revelations},
	the Cauchy problem \eqref{CPintro} is $\mathscr{S}$-very weakly well-posed, i.e. \eqref{CPintro} admits a $\mathscr{S}$ very weak solution and the solution is unique in the following sense: $\mathscr{S}$-negligible changes on the regularisations of the equation coefficients and $\mathscr{S}$-negligible changes on the regularisations of the initial data lead to $\mathscr{S}$-negligible changes in the corresponding $\mathscr{S}$-very weak solution.
\end{Th}

\begin{Rem}
	Theorem \ref{main_theorem} implies that \eqref{CPintro} is well-posed in $\mathcal{G}_{\mathscr{S}}$. More precisely, we can embed the coefficients $c_j$ of \eqref{city_of_tears} (via regularisations parametrised by a suitable positive scale $\omega(\varepsilon)$, see \eqref{omegaepsilon}) and the Cauchy data of \eqref{CPintro} in the Colombeau algebra $\mathcal{G}_{\mathscr{S}}$ in such a way that the corresponding embedded Cauchy problem is well-posed in $\mathcal{G}_{\mathscr{S}}$.
\end{Rem}

To prove Theorem \ref{main_theorem} we need to obtain a $\mathscr{S}$-very weak solution for \eqref{CPintro}, that is, we need to solve the regularised problem \eqref{regularised_CPintro} in $\mathscr{S}(\R^{n})$. In order to do it, we shall employ the classical techniques developed in \cite{KB}. In the sequel we shall fix the considered regularisations for the coefficients of \eqref{city_of_tears} and for the data in \eqref{CPintro}. In what follows, we shall use the standard positive scale $\varepsilon$ for sake of simplicity and we shall replace it with a suitable positive scale $\omega(\varepsilon)$ just when it is needed. 

\subsection{Regularisation of $c_j(t,x)$, $j=0,1,\ldots,n$} 
Let $\psi \in \mathscr{S}(\R^n)$ with $\int \psi = 1$ and $\phi \in \mathscr{S}(\R^{n})$ with $\phi(0) = 1$. We then define 
\begin{equation}\label{eq_estimates_regularised_c_j_0}
	c_{j,\varepsilon}(t,x) := \phi(\varepsilon x)(\psi_\varepsilon \ast_x c_j)(t,x). 
\end{equation}
Then we have $c_{j,\varepsilon} \in C([0,T];\mathscr{S}(\R^{n}))$ and the following estimate holds
\begin{equation}\label{regbj}
	|\langle x \rangle^{M} \partial^{\beta}_{x}c_{j,\varepsilon}(t,x)| \leq C_{\beta,M} \varepsilon^{-|\beta|-M-N}, \quad t \in [0,T], x \in \R^n, \beta \in \N_0^n,
\end{equation}
where $N > 0$ is a number depending on the coefficients $c_j$, $j = 0, 1, \ldots, n$, and on the dimension.

\subsection{Regularisation of the Cauchy data} 
Let $\mu, \nu \in \mathscr{S}(\R^{n})$ with $\int \mu = 1$ and $\nu(0) = 1$. We then define $f_\epsilon(t,x) = \nu(\varepsilon x) (\mu_{\varepsilon} \ast_x f)(t,x)$ and $g_\epsilon(x) = \nu(\varepsilon x) \mu_{\varepsilon}(x) \ast g(x)$. Then we have $f_{\varepsilon} \in C([0,T]; \mathscr{S}(\R^{n}))$ and the following estimates
\begin{equation}
	|\langle x \rangle^{M} \partial^{\beta}_{x}f_{\varepsilon}(t,x)|\leq C_{\beta,M} \varepsilon^{-|\beta|-M-\tilde{N}_f},\quad 
	|\langle x \rangle^{M} \partial^{\beta}_{x} g_\epsilon(x)| \leq C_{\beta,M} \varepsilon^{-|\beta|-M-\tilde{N}_g},
\end{equation}
where $\tilde{N}_f > 0$ is a number depending on $f$ and on the dimension $n$, $\tilde{N}_g > 0$ is a number depending on $g$ and on the dimension $n$. By these estimates we immediately get that for all $m,M \in \N_0$ there exist $C>0$, $N_f\in\N_0$ and $N_g\in\N_0$ such that 
\begin{eqnarray}\label{stimaf}
	\| f_\varepsilon(t,\cdot) \|_{H^{m,M}}&\leq& C_{m,M} \varepsilon^{-N_f}, \quad \forall\, t\in[0,T],\ \epsilon\in (0,1] \\
	\| g_\varepsilon \|_{H^{m,M}}&\leq& C_{m,M} \varepsilon^{-N_g},\quad \forall\, \epsilon\in (0,1].\label{stimag}
\end{eqnarray}

\medskip

We are finally ready to define the family of regularised Cauchy problems that we shall study in the subsequent sections. We consider the family of regularised operators \eqref{regularised_opintro} for $\varepsilon \in (0, 1]$ 
and then the family of regularised Cauchy problems
\begin{equation}\label{regularised_CP_true}
	\begin{cases}
		S_{\varepsilon} v(t,x) = f(t,x), \quad t \in [0,T],\, x \in \R^{n}, \\
		v(0,x) = g(x), \quad \quad \,\, \, x \in \R^{n},
	\end{cases}
\end{equation}
where $f \in C([0,T];\mathscr{S}(\R^{n}))$, $g \in \mathscr{S}(\R^{n})$ and the coefficients of $S_\varepsilon$ defined in \eqref{regularised_opintro} are given in \eqref{eq_estimates_regularised_c_j_0}. In the next sections we will obtain a net of solutions $(u_{\varepsilon})_{\varepsilon \in (0,1]}$ where for every $\varepsilon$ the function $u_{\varepsilon} \in C([0,T]; \mathscr{S}(\R^{n}))$  is the unique solution of the Cauchy problem \eqref{regularised_CP_true}. Moreover, we will also derive energy estimates for the solutions $u_{\varepsilon}$ expliciting how the constants depend on the parameter $\varepsilon$.

\smallskip


\section{Solving the regularised problem \eqref{regularised_CP_true}}\label{n}
\

We want to prove that there exists a unique solution in $\mathscr{S}(\R^{n})$ for the problem \eqref{regularised_CP_true}. Since $c_{j,\varepsilon} \in C([0,T];\mathscr{S}(\R^{n}))$ we know that \eqref{regularised_CP_true} is well-posed in $H^{\infty}(\R^{n})$ without loss of derivatives, i.e., there exists a unique solution $u \in C([0,T];H^{\infty}(\R^{n}))$ satisfying (cf. Proposition $5.1$ in \cite{AACG})
$$ 
\|u_{\varepsilon}\|^{2}_{H^m} \leq  C_{n,m} \exp\left\{ C_{T,m,n} e^{\varepsilon^{-N-\theta_{n,m}}} \right\} \left\{ \|g\|^{2}_{H^{m}} + \int_{0}^{t} \|f(\tau)\|^{2}_{H^{m}} d\tau \right\},
$$
where $N$ is a natural number depending on the coefficients $c_j$ and $\theta_{n,m}$ is a natural number depending on the dimension $n$ and on the Sobolev index $m$. We want to conclude that the solution $u$ is also rapidly decreasing at infinity. This motivates the following conjugation: for a fixed $s \in \N_0$ we define 
\begin{align}\label{o_toco}
S_{\varepsilon, s} &:= \langle x \rangle^{s} S_{\varepsilon} \langle x \rangle^{-s} \nonumber \\
	&=	D_t + \sum_{j=1}^{n} D^{2}_{x_j} + \sum_{j=1}^{n} c_{j,\varepsilon}(t,x) D_{x_j} + 2is \langle x \rangle^{-1} \sum_{j=1}^{n}\frac{x_j}{\langle x \rangle} D_{x_j} + c_{0,\varepsilon, s}(t,x),
\end{align}
where 
\begin{align*}
	c_{0,\varepsilon, s}(t,x)  &= c_{0,\varepsilon}(t,x) + is \langle x \rangle^{-1} \sum_{j=1}^{n} c_{j,\varepsilon}(t,x) \frac{x_j}{\langle x \rangle} 
	- s(s+2) \langle x \rangle^{-2} \frac{|x|^{2}}{\langle x \rangle^{2}} +  ns \langle x \rangle^{-2}. 
\end{align*}

Observe that in $S_{\varepsilon, s}$ the new purely imaginary first order term
\begin{align}\label{scars_of_time}
	2is \langle x \rangle^{-1} \sum_{j=1}^{n}\frac{x_j}{\langle x \rangle} D_{x_j}
\end{align}
appears. This term decays exactly like $\langle x \rangle^{-1}$ for $|x| \to \infty$. From the classical theory presented in \cite{KB} we know that the Cauchy problem for the operator $S_{\varepsilon, s}$ is well-posed in $H^{\infty}(\R^{n})$ and the decay $\langle x \rangle^{-1}$ on the first order coefficients produces a finite loss of Sobolev regularity with respect to the initial data. Since \eqref{scars_of_time} does not depend on $\varepsilon$, the loss of regularity produced by this new term will not depend on $\varepsilon$, but will depend on $s$. This intutitive idea leads to the following result.  

\begin{Prop}\label{pompem}
	Let $\tilde{f} \in C([0,T]; \mathscr{S}(\R^{n}))$ and $\tilde{g} \in \mathscr{S}(\R^{n})$. There exists a unique solution $u$ in $C([0,T];H^{\infty}(\R^{n}))$ to the Cauchy problem 
	\begin{equation}\label{sign_of_the_cross}
		\begin{cases}
			S_{\varepsilon,s} u(t,x) = \tilde{f}(t,x), \quad t \in [0,T], x \in \R^{n}, \\
			u(0,x) = \tilde{g}(x), \qquad \quad \, 	x \in \R^{n},
		\end{cases}
	\end{equation}
	and $u$ satisfies: for each $m \in \N_0$
	\begin{equation}\label{psycho_killer}
		\| u \|^{2}_{H^{m-c(s)}} \leq  C_{m,s,n,T} \exp\left\{ C_{m,s,n,T} e^{\varepsilon^{-J(N,s,m,n)}} \right\} \left\{ \|\tilde{g}\|^{2}_{H^{m}} + \int_{0}^{t} \|\tilde{f}(\tau)\|^{2}_{H^{m}} d\tau \right\},
	\end{equation}
	where $c(s) = Cs$ for some constant $C > 0$ and $J(N, s, m,n)$ is a natural number depending on $N,s,m$ and $n$.
\end{Prop}

Proposition \ref{pompem} is crucial to prove the main result of the paper. Its proof is the most challenging part of the paper. However, since it is long and technical, we postpone it to Section \ref{Iron_maiden} in order to address the reader as fast as possible to the proof of the main result.

\smallskip

The next theorem is a direct consequence of Proposition \ref{pompem}.

\begin{Th}\label{th}\label{apprn}
	For every $\varepsilon \in (0,1]$ 
	consider the regularised Cauchy problem \eqref{regularised_CP_true} with initial data $f \in C([0,T];\mathscr{S}(\R^{n}))$ and $g \in \mathscr{S}(\R^{n})$. Then there exists a unique solution $u_{\varepsilon} \in C([0,T]; \mathscr{S}(\R^{n}))$ for the problem \eqref{regularised_CP_true}. Moreover, the solution $u_\varepsilon$ satisfies for every $m,s\in\N_0$:
	\begin{equation}\label{national_acrobat}
		\|u_{\varepsilon}\|^{2}_{H^{m-c(s),s}} \leq  \tilde{C}_{m,s,n,T} \exp\left\{ C_{m,s,n,T} e^{\varepsilon^{-J(N,s,m,n)}} \right\} \left\{ \|g\|^{2}_{H^{m,s}} + \int_{0}^{t} \|f(\tau)\|^{2}_{H^{m,s}} d\tau \right\},
	\end{equation}
	where 
	\begin{itemize}
		\item[(i)] The constants $C_{m,s,n,T}$ and $\tilde{C}_{m,s,n,T}$ depend on the coefficients $c_j$, $j=0,1,\ldots,n$, and on the mollifiers $\psi, \phi$; 
		
		
		\item[(ii)] $N$ is a positive integer depending on the coefficients $c_0, c_1, \ldots, c_n$ and on the dimension;
		
		\item[(iii)] $J(N,s,m,n)$ is a natural number depending on $N,s,m$ and $n$.
	\end{itemize}
\end{Th}

\smallskip


\section{From the regularised problem to the original problem}\label{Schwartz very weak well-posedness}
\

In this section we shall apply Theorem \ref{th} to prove Theorem \ref{main_theorem}.

\subsection{Existence}
The regularised Cauchy data $f_\epsilon (t), g_\epsilon$ fulfill estimates \eqref{stimaf}, \eqref{stimag}; moreover, from Theorem \ref{apprn} the regularised problem \eqref{regularised_CP_true} with data $f_\varepsilon(t)$ and $g_\varepsilon$ has a unique solution $u_\varepsilon\in C([0,T];\mathscr S(\R^{n}))$ satisfying for every $m, s\in\N_0$ 
\begin{eqnarray*}
\|u_{\varepsilon}\|^{2}_{H^{m-c(s),s}}& \leq&  \tilde{C}_{m,s,n,T} \exp\left\{ C_{m,s,n,T} e^{\varepsilon^{-J(N,s,m,n)}} \right\} \left\{ \|g_\epsilon\|^{2}_{H^{m,s}} + \int_{0}^{t} \|f_\epsilon(\tau)\|^{2}_{H^{m,s}} d\tau \right\}, \\
	&\leq& A_{m,s}  \,\text{exp}(B_{m,s}e^{\varepsilon^{-N_{m,s}}}) \varepsilon^{-(N_f + N_g)}, \quad \varepsilon \in (0,1], t \in [0, T],
\end{eqnarray*}
where $A_{m,s}, B_{m,s} > 0$ are independent from $\varepsilon$ and $N_{m,s}\in \N_0$ is independent from $\varepsilon$ but depends on the coefficients $c_j$, $j = 0, 1, \ldots, n$, on the dimension $n$ and on Sobolev indices $m,s$.  
We thus obtain a $\mathscr{S}$-moderate net of solutions for the problem \eqref{CPintro} by regularising the coefficients $c_j$ via mollifiers indexed by the positive scale 
\beqs\label{omegaepsilon}
\omega(\varepsilon) = \{ \log(\log(\log(\log(\varepsilon^{-1})))) \}^{-1}, \quad \varepsilon \in (0, 1).
\eeqs
Indeed, by inequality $\log(y) \leq C_{r} y^{\frac{1}{r}}$, for all $y \geq 2$ and all $r \geq 1$, we get
$$\omega(\varepsilon)^{-N_{m,s}} = \{\log(\log(\log(\log(\varepsilon^{-1}))))\}^{N_{m,s}} \leq C_{N_{m,s}}^{N_{m,s}} \log(\log(\log(\varepsilon^{-1}))) = \log\left(\log(\log(\varepsilon^{-1}))^{C^{N_{m,s}}_{N_{m,s}}}\right)$$
and so we have $$ B_{m,s}e^{\omega(\varepsilon)^{-N_{m,s}}} \leq B_{m,s} \log(\log(\varepsilon^{-1}))^{C^{N_{m,s}}_{N_{m,s}}}\leq \tilde B_{m,s, N_{m,s}}\log(\varepsilon^{-1}) = \log(\varepsilon^{-D_{N_{m,s}}}),
$$
where $D_{N_{m,s}} \in \N_0$ is a large number depending on $N_{m,s}$ and on $B_{m,s}$. In this way, we obtain that for any $m,s \in \N_0$ we find constants $A_{m,s}, D_{N_{m,s}}$ such that 
$$
\|u_{\varepsilon}(t)\|^{2}_{H^{m-c(s),s}} \leq A_{m,s} \varepsilon^{-D_{N_{m,s}}-(N_f+N_g)}, \quad t \in [0,T], \, \varepsilon \in (0,1). 
$$
Therefore, the Cauchy problem \eqref{CPintro} admits 
a $\mathscr S'$-very weak solution.

\subsection{Uniqueness}\label{subsection_uniqueness}

Suppose that we perturb the coefficients of the regularised operator $S_{\varepsilon}$ 
in the following way
\begin{align*}
	S'_{\varepsilon} := D_t + \sum_{j=1}^{n} D^{2}_{x_j} + \sum_{j=1}^{n}(c_{j,\varepsilon}(t,x)+c'_{j,\varepsilon}(t,x)) D_{x_j} + (c_{0,\varepsilon}(t,x)+c'_{0,\varepsilon}(t,x))
\end{align*}
where $(c'_{j,\varepsilon})_{\varepsilon}$ are $\mathscr{S}$-negligible.
Suppose moreover that the net $(u'_{\varepsilon})_\varepsilon$ solves the Cauchy problem
\[
\begin{cases}
	S'_{\varepsilon} u(t,x) = f_\varepsilon(t,x) +p_{\varepsilon}(t,x), \quad t \in [0,T],\, x \in \R, \\
	u(0,x) = g_\varepsilon(x) + q_{\varepsilon}(x) \quad\qquad \,\,\,\,\, x \in \R,
\end{cases}
\]
where $(p_\varepsilon)_{\varepsilon}$ and $(q_\varepsilon)_{\varepsilon}$ are $\mathscr{S}$-negligible nets. We now want to compare the two very weak solutions $u_\varepsilon$ and $u'_\varepsilon$. Since
\[
\begin{cases}
	S_{\varepsilon} u_\varepsilon(t,x) = f_{\varepsilon}(t,x), \quad t \in [0,T],\, x \in \R, \\
	u_\varepsilon(0,x) = g_{\varepsilon}(x), \quad \quad \,\,\, x \in \R,
\end{cases}
\]
and
\[
\begin{cases}
	S'_{\varepsilon} u'_\varepsilon(t,x) = f_{\varepsilon}(t,x)+p_{\varepsilon}(t,x), \quad t \in [0,T],\, x \in \R, \\
	u'_\varepsilon(0,x) = g_{\varepsilon}(x) + q_{\varepsilon}(x),  \qquad\quad\,\, x \in \R,
\end{cases}
\]
it follows that 
\[
\begin{cases}
	S_\varepsilon(u_\varepsilon-u'_\varepsilon)(t,x) = -p_{\varepsilon}(t,x)-(S_{\varepsilon}-S'_\varepsilon) u'_\varepsilon(t,x), \quad t \in [0,T],\, x \in \R, \\
	(u_\varepsilon-u'_\varepsilon)(0,x) =-q_{\varepsilon}(x), \qquad \qquad  \qquad \qquad \qquad \quad x \in \R.
\end{cases}
\]
By applying estimate \eqref{national_acrobat}, using the scale \eqref{omegaepsilon} to regularise the coefficients $c_j$, we have that for all $m, s \in \N_0$ there exist $\tilde{C}>0$ and $\tilde{N}\in\N_0$ such that 
\begin{equation}
	\label{est_neg_uniq}
	\|u_{\varepsilon}-u'_\varepsilon\|^{2}_{H^{m-c(s),s}} \leq  \tilde{C}\varepsilon^{-\tilde{N}} 
	\left\{ \|q_{\varepsilon}\|_{H^{m,s}} +  \int_{0}^{t} \|p_{\varepsilon}(\tau)+(S_{\varepsilon}-S'_\varepsilon) u'_\varepsilon(\tau)\|^{2}_{H^{m,s}} d\tau\right\}.
\end{equation}
Since $(p_\varepsilon)$ and $(q_\varepsilon)$ are both $\mathscr{S}$-negligible, the coefficients of the operator $(S_\varepsilon-S'_\varepsilon)_\varepsilon$ are $\mathscr{S}$-negligible and $(u'_{\varepsilon})_{\varepsilon}$ is $\mathscr{S}$-moderate, we conclude that the right-hand side of \eqref{est_neg_uniq} can be estimated by any positive power of $\varepsilon$. Hence, $(u_{\varepsilon}-u'_{\varepsilon})_{\varepsilon}$ is $\mathscr{S}$-negligible. 

%

\smallskip


\section{Proof of Proposition \ref{pompem}}\label{Iron_maiden}
\

To solve \eqref{sign_of_the_cross}, we take inspiration from the approach used in \cite{KB} which is based on a suitable change of variable which turns \eqref{sign_of_the_cross} into an auxiliary Cauchy problem well-posed in $L^2(\R^n)$. The change of variable we use is the composition of two  pseudodifferential operators. In the next lines we introduce this change of variable. A survival toolkit of what we need about the theory of pseudodifferential operators is reported in the Appendix \ref{pseudodifferential_operators} at the end of the paper.
 

In \cite{KB} the authors proved that for every $M_1, M_2>0, $ there exist real-valued functions $\tilde\lambda_1, \tilde\lambda_2$ satisfying for $x \in \R^n, |\xi|\geq 1, \alpha, \beta \in \N_0^n$, $\beta \neq 0$: 
$$
|\partial^{\alpha}_{\xi}\tilde{\lambda}_{1}(x,\xi)| \leq
	M_{1}C_{\alpha} |\xi|^{-|\alpha|} \log(2+|\xi|),$$
$$|\partial^{\alpha}_{\xi}\tilde{\lambda}_{2}(x,\xi)| \leq	M_{2} C_{\alpha} |\xi|^{-|\alpha|},$$
 $$
|\partial^{\alpha}_{\xi} \partial^{\beta}_{x} \tilde{\lambda}_{\ell}(x,\xi)| \leq M_{\ell} C_{\alpha,\beta} |\xi|^{-|\alpha|}, \quad \ell=1,2,
$$ and 
$$
\sum_{j=1}^{n} \xi_j \partial_{x_j} \tilde{\lambda}_{\ell}(x,\xi) \leq -M_{\ell} \langle x \rangle^{-\ell} \chi\left( \frac{\langle x \rangle}{|\xi|} \right) |\xi|, \qquad \ell=1,2,
$$
where $\chi \in \mathcal{C}^{\infty}_{c}(\R)$ such that $\chi(t) = 0$ for $|t| > 1$, $\chi(t) = 1$ for $|t| < \frac{1}{2}$, $t \chi'(t) \leq 0$ and $0 \leq \chi \leq 1$.

For a large parameter $h \geq 1$ to be chosen later on, we consider 
\begin{equation}\label{lambdan}
	\lambda_{\ell}(x,\xi) = \tilde{\lambda}_\ell(x,\xi) (1-\chi)\left( \frac{ |\xi| }{ h } \right), \qquad \ell=1,2.
\end{equation}
Then, since $\langle \xi \rangle_{h} = \sqrt{ h^2 + |\xi|^2 } \leq \sqrt{5}h$ on the support of $(1-\chi)(h^{-1}|\xi|)$, we have
$$
|\partial^{\alpha}_{\xi} \partial^{\beta}_{x} \lambda_{\ell}(x,\xi)| \leq
\begin{cases}
 M_{1} C_{\alpha} \log(\langle \xi \rangle_h) \langle \xi \rangle^{-|\alpha|}_{h}, \quad \ell = 1, \beta = 0,\\
 M_{2} C_{\alpha} \langle \xi \rangle^{-|\alpha|}_{h}, \quad \ell=2, \beta = 0, \\
 M_{\ell} C_{\alpha,\beta} \langle \xi \rangle^{-|\alpha|}_{h}, \quad \ell= 1,2, \beta \neq 0,
 \end{cases}
$$
where the constants $C_\alpha$ and $C_{\alpha,\beta}$ do not depend on $M_{\ell}$ and on $h$. To define the change of variable, we shall use the symbols $e^{\lambda_\ell(x,\xi)}, \ell =1,2.$
The following estimates follow by a standard application of Fa\`a di Bruno formula:
$$
|\partial^{\alpha}_{\xi}e^{\pm \lambda_1(x,\xi)}| \leq C_{\alpha}e^{M_1}  (\log \langle \xi \rangle_h)^{|\alpha|} \langle \xi \rangle^{- |\alpha|}_{h} e^{\pm \lambda_1(x,\xi)},
$$
$$
|\partial^{\alpha}_{\xi}\partial^{\beta}_{x}e^{\pm \lambda_1(x,\xi)}| \leq C_{\alpha,\beta} e^{M_1} \langle \xi \rangle^{- |\alpha|}_{h} e^{\pm \lambda_1(x,\xi)}.
$$
On the other hand,
$$
e^{\pm \lambda_1}(x,\xi) \leq e^{M_1C_{0} \log \langle \xi \rangle_h} \leq \langle \xi \rangle^{C_0M_1}.
$$
Hence it follows that there exists a number $c > 0$ such that
$$
e^{\pm \lambda_1} \in S^{cM_1}(\R^{2n})
$$
and 
$$
|\partial^{\alpha}_{\xi}\partial^{\beta}_{x}e^{\pm \lambda_1(x,\xi)}| \leq C_{\alpha,\beta} e^{M_1} \langle \xi \rangle^{cM_1 - |\alpha|}_{h}.
$$
Concerning $e^{\lambda_2(x,\xi)}$ we have that it belongs to $S^0(\R^{2n}).$

We shall use the operators $e^{\lambda_\ell}(x,D)$ to handle the first order terms of the operator $S_{\varepsilon,s}$ defined in \eqref{o_toco}. These terms  are given by
\begin{align}\label{first_order_term_1}
	\sum_{j=1}^{n} c_{j,\varepsilon}(t,x) D_{x_j}
\end{align}
and
\begin{align}\label{first_order_term_2}
	2is \langle x \rangle^{-1} \sum_{j=1}^{n}\frac{x_j}{\langle x \rangle} D_{x_j}.
\end{align}
Since $c_{j,\varepsilon}(t,x)$ decays faster than $\langle x \rangle^{-1}$, we have that \eqref{first_order_term_1} does not produce any loss of Sobolev regularity. On the other hand, the term \eqref{first_order_term_2} decays exactly as $\langle x \rangle^{-1}$, so it produces a loss of regularity, but this loss will be independent on the parameter $\varepsilon$, since \eqref{first_order_term_2} does not depend on $\varepsilon$. Our idea is then to use $e^{\lambda_{2}}(x,D)$ to handle the term \eqref{first_order_term_1} and $e^{\lambda_{1}}(x,D)$ to control the new first order term \eqref{first_order_term_2} produced by the conjugation by $\langle x \rangle^{s}$. Namely, we shall consider a change of variable of the form
\begin{equation}\label{auxiliary_important_CP}
	\begin{cases}
		e^{\lambda_2} \circ e^{\lambda_1} \circ S_{\varepsilon,s} \circ \{e^{\lambda_1}\}^{-1} \circ \{e^{\lambda_2}\}^{-1} \circ e^{\lambda_2} \circ e^{\lambda_1} u(t,x) = e^{\lambda_2} \circ e^{\lambda_1}\tilde{f}(t,x), \\
		e^{\lambda_{2}} \circ e^{\lambda_1} u(0,x) = e^{\lambda_2} \circ e^{\lambda_1} \tilde{g}(x).
	\end{cases}
\end{equation}

\begin{Rem}
	Notice that in the conjugation the leading term is $e^{\lambda_1}(x,D)$. However,  we need to perform two conjugations in order to avoid a loss of Sobolev regularity depending on $\varepsilon$ which would appear if we consider the sole transformation $e^{\lambda_1}(x,D)$.
\end{Rem}

In the remaining part of the present section we shall derive a priori energy estimates to the auxiliary problem \eqref{auxiliary_important_CP}, see Proposition \ref{sex_pistols}. Then the proof of Propostion \ref{pompem} will be consequence of the derived energy estimates and the mapping properties of $e^{\lambda_1}(x,D)$ and $e^{\lambda_2}(x,D)$.

\subsection{Invertibility of $e^{\lambda_1}(x,D)$ and conjugation theorem}
\

The reverse operator $^{R}\{e^{\lambda_1}(x,D)\}$ was introduced in \cite{KW} (see Proposition $2.13$) as the transpose operator of $e^{\lambda_1}(x,-D)$. Since $\lambda_1$ is real-valued, in this particular case, the reverse operator coincides with the $L^2$ adjoint of $e^{\lambda_1}(x,D)$. We have that 
$$
e^{\lambda_1}(x,D) \circ\, ^{R}\{e^{-\lambda_1}(x,D)\} = \mu(x,D),
$$
where 
$$
\mu(x,\xi) = Os- \iint e^{-iy\eta} e^{\lambda_1(x, \xi+\eta) - \lambda_1(x+y,\xi+\eta)} dy\dslash\eta.
$$
So, Taylor's formula and usual computations give
\begin{align}\label{blood_brothers}
	\mu(x,\xi) = \sum_{|\alpha| < N} \frac{1}{\alpha!} \partial^{\alpha}_{\xi}\{e^{\lambda_1(x,\xi)} D^{\alpha}_{x} e^{-\lambda_1(x,\xi)}\} + r_{N}(x,\xi),
\end{align}
where 
\begin{align*}
	r_N(x,\xi) = N \sum_{|\gamma| = N} \int_0^1 (1-\theta)^{N-1}  Os- \iint e^{-iy\eta} \partial^{\alpha}_{\xi} \{e^{\lambda_1(x, \xi+\theta\eta)} D^{\gamma}_{x}e^{-\lambda_1(x+y,\xi+\theta\eta)} \} dy\dslash\eta \, d\theta.
\end{align*}
The advantage to use the reverse operator is that in the above asymptotic formula the terms
$$
\partial^{\alpha}_{\xi}\{e^{\lambda_1(x,\xi)} D^\alpha_x e^{-\lambda_1(x,\xi)}\}
$$
appear, and they are of order exactly $-|\alpha|$ (i.e. the $x$-derivative in the above asymptotic formula allow us to avoid the $\log$ growth appearing in the symbol estimates of $\lambda_1$). The symbols $e^{\pm \lambda_1}$ are of finite order, so we can estimate $r_N$ in the following way
$$
|\partial^{\alpha}_{\xi} \partial^{\beta}_{x} r_{N}(x,\xi)| \leq C_{\alpha, \beta, N, M_1,n} \langle \xi \rangle^{2cM_1 - N}_h.
$$
Since \eqref{blood_brothers} holds for any $N \in \N$, it follows that 
$$
e^{\lambda_1}(x,D) \circ\, ^{R}\{e^{-\lambda_1}(x,D)\} = I - \textrm{op}\left(i\sum_{j=1}^{n} \partial_{\xi_j} \{\partial_{x_j} \lambda_1(x,\xi)\} \right) + r_{-2}(x,D) = I + r_{-1}(x,D),
$$
where $r_{-2}$ has order $-2$. Hence, if we assume that $h \geq h_0(M_1)$ we obtain that 
\begin{eqnarray} \nonumber
\{e^{\lambda_1}(x,D) \}^{-1} &=&  ^{R}\{e^{-\lambda_1}(x,D)\} \circ \sum_{j \geq 0} (-r(x,D))^j \\ &=& ^{R}\{e^{-\lambda_1}(x,D)\} \circ \textrm{op}\left( 1 - i \sum_{j=1}^{n} \partial_{\xi_j}\partial_{x_j}\lambda_1(x,\xi) + r_{-2}(x,\xi) \right) , \label{alessia}
\end{eqnarray}
where $\sum_{j \geq 0} (-r)^j$ has symbol of order zero and $r_{-2}$ has symbol of order $-2$ satisfying 
$$
|\partial^{\alpha}_{\xi} \partial^{\beta}_{x} r_{-2}(x,\xi)| \leq C_{\alpha, \beta,M_1, n} \langle \xi \rangle^{-2-|\alpha|}_{h},
$$
for some $C > 0$ independent on $M_1$ and $h$ (cf. Theorem $I.1$ at page $372$ of \cite{Kumano-Go}).

For a $p \in S^{m}(\R^{2n})$ such that 
$$
|\partial^{\alpha}_{\xi} \partial^{\beta}_{x} p(x,\xi)| \leq C_{\alpha, \beta} \langle \xi \rangle^{m-|\alpha|}_{h},
$$
we consider
$$
|p|^{(m)}_{\ell} = \max_{|\alpha|, |\beta| \leq \ell} \sup_{x,\xi \in \R^n} |\partial^{\alpha}_{\xi} \partial^{\beta}_{x} p(x,\xi)| \langle \xi \rangle^{-m+|\alpha|}_{h}.
$$
Thus
$$
|\partial^{\alpha}_{\xi} \partial^{\beta}_{x} p(x,\xi)| \leq |p|^{(m)}_{|\alpha+\beta|} \langle \xi \rangle^{m-|\alpha|}_{h}.
$$
Now we need to understand the composition 
$
e^{\lambda_1}(x,D) \circ p(x,D) \circ\, \{e^{\lambda_1}(x,D)\}^{-1} .
$
This is done in two steps in view of \eqref{alessia}. The next theorem is a direct consequence of Theorem \ref{Hotel_california}.

\begin{Th}\label{quase_sem_querer}
 	For any $N \in \N$ the symbol of $e^{\lambda_1}(x,D)\circ p(x,D) \circ\, ^{R}\{ e^{-\lambda_1}(x,D)\}$ is given by
 	\begin{equation} \label{symbolalessia}
 	p(x,\xi)+\sum_{1\leq |\alpha+\beta|<N} \frac{1}{\alpha!\beta!} \partial^{\beta}_{\xi} \{ \partial^{\alpha}_{\xi}e^{\lambda_1(x,\xi)} D^{\alpha}_{x}p(x,\xi) D^{\beta}_{x}e^{-\lambda_1(x,\xi)} \} + r_{N}(x,\xi),
 	\end{equation}
 	where 
 	$$
 	|\partial^{\alpha}_{\xi} \partial^{\beta}_{x} r_{N}(x,\xi)| \leq C_{\alpha, \beta, N, M_1, n} |p|^{(m)}_{|\alpha+\beta|+N+J} \langle \xi \rangle_h^{2cM_1 + m - N},
 	$$
 	$J$ being a natural number depending on $M_1, m$ and on the dimension $n$.
\end{Th}

	Theorem \ref{quase_sem_querer} implies the following: if we know that $p_\varepsilon$ satisfies
	$$
	|\partial^{\alpha}_{\xi} \partial^{\beta}_{x} p_\varepsilon(x,\xi)|	\leq C_{\alpha, \beta} \varepsilon^{-N_1 - |\beta|} \langle \xi \rangle^{1-|\alpha|}_{h},
	$$
	then taking $N(M_1) = \lfloor 2cM_1 + 1 \rfloor + 1$ we get that $r_N$
 is a symbol of order zero and satisfies 
	\begin{equation}\label{needles_and_pins}
	|\partial^{\alpha}_{\xi} \partial^{\beta}_{x} r_{N}(x,\xi)| \leq C_{\alpha, \beta, M_1, n} \varepsilon^{-N_1 - |\beta| - J(M_1,n)} \langle \xi \rangle_h^{-|\alpha|},
	\end{equation}
	where $J$ is a natural number depending on $M_1$ and $n$. Thus, \eqref{symbolalessia} becomes
	$$
	p_\varepsilon(x,\xi) + \underbrace{\sum_{1\leq |\alpha+\beta|<N} \frac{1}{\alpha!\beta!} \partial^{\beta}_{\xi} \{ \partial^{\alpha}_{\xi}e^{\lambda_1(x,\xi)} D^{\alpha}_{x}p_{\varepsilon}(x,\xi) D^{\beta}_{x}e^{-\lambda_1(x,\xi)} \}}_{q_{\varepsilon}(x,\xi)} + r_{0}(x,\xi),
	$$
	where $r_0$ satisfies \eqref{needles_and_pins}. Now note that we can write $q_\varepsilon$ in the following way:
	\begin{align*}
	q_{\varepsilon} &= \sum_{j=1}^{n} \partial_{\xi_j} \lambda_1 D_{x_j} p_{\varepsilon} -\sum_{j=1}^{n} \partial_{\xi_j}\{p_\varepsilon D_{x_j}\lambda_1\}  + \sum_{2 \leq |\alpha+\beta|<N} \frac{1}{\alpha!\beta!} \partial^{\beta}_{\xi} \{ \partial^{\alpha}_{\xi}e^{\lambda_1(x,\xi)} D^{\alpha}_{x}p_{\varepsilon}(x,\xi) D^{\beta}_{x}e^{-\lambda_1(x,\xi)} \} \\
	&= \sum_{j=1}^{n} \partial_{\xi_j} \lambda_1 D_{x_j} p_{\varepsilon} + r_0,
	\end{align*}
	for a new zero order term $r_0$ satisfying an estimate like \eqref{needles_and_pins}. Hence 
	\begin{equation}\label{palhas_do_coqueiro}
	e^{\lambda_1}(x,D) \circ p_\varepsilon(x,D) \circ\, ^{R} \{e^{-\lambda_1}(x,D)\} = p_\varepsilon(x,D) + \textrm{op}\left(\sum_{j=1}^{n} \partial_{\xi_j} \lambda_1 D_{x_j} p_{\varepsilon}\right) + r_{0}(x,D).
	\end{equation}
	Since
	$
	\sum\limits_{j=1}^{n} \partial_{\xi_j} \lambda_1 D_{x_j} p_{\varepsilon}
	$
	behaves like $\varepsilon^{-N_1-1} \log\langle \xi \rangle_{h}$, it cannot be regarded as a zero order term.

\subsection{Invertibility of $e^{\lambda_2}(x,D)$}
\

\vskip0.2cm

Here we just report results derived in \cite{AACG} at Subsection $6.1$.

\begin{Lemma}\label{die_young_sabbath}
	For all $h \geq h_0(M_2,n) := A e^{(2C+1)M_2}$, for some $A, C > 0$ independent on $M_2$, the operator $e^{\lambda_2}(x,D)$ is invertible and we have
	$$
	\left(e^{\lambda_2}(x,D)\right)^{-1} = e^{-\lambda_2}(x,D) \circ \sum_{j \geq 0} (-r(x,D))^{j}, 
	$$
	where 
	$$
	r(x,\xi) = i\sum_{k=1}^{n} \partial_{\xi_k} \lambda_2(x,\xi)\partial_{x_k}\lambda_2(x,\xi) + r_{-2}(x,\xi).
	$$ and 
	\begin{equation}\label{the_writing_on_the_wall}
		r_{-2}(x,\xi) = \sum_{|\gamma|=2} \frac{2}{\gamma!} Os - \iint e^{-y\eta} \int_{0}^{1} (1-\theta) \partial^{\gamma}_{\xi} e^{\lambda_2(x,\xi+\theta\eta)} d\theta D^{\gamma}_{x} e^{-\lambda_2(x+y,\xi)} dy\dslash\eta.
	\end{equation}
	Morever, 
	$$
	\sum_{j \geq 0} (-r(x,D))^{j} = s(x,D),
	$$
	where $s(x,\xi)$ is a zero order symbol and its seminorms do not depend on $h$ and on $M_2$.
\end{Lemma}

We need to write $s(x,D)$ in the following convenient way:
\begin{align*}
	s(x, D) &= I - r(x,D) + \sum_{j \geq 2} (-r(x,D))^{j} = I - r(x,D) + (r(x,D))^{2} s(x,D)  \\
	&= I - \textrm{op}\left( i\sum_{k=1}^{n} \partial_{\xi_k} \lambda_2(x,\xi)\partial_{x_k}\lambda_2(x,\xi) \right) - r_{-2}(x,D) +  (r(x,D))^{2} s(x,D).
\end{align*}
So, 
\begin{equation}\label{hell_on_Earth}
	s(x,\xi) = 1 - i\sum_{k=1}^{n} \partial_{\xi_k} \lambda_2(x,\xi)\partial_{x_k}\lambda_2(x,\xi) + s_{-2}(x,\xi),
\end{equation}
where $s_{-2}(x,\xi)$ has order $-2$ and satifies 
$$
|\partial^{\alpha}_{\xi}\partial^{\beta}_{x} s_{-2}(x,\xi)| \leq C_{\alpha,\beta,n} e^{CM_2} \langle \xi \rangle^{-2-|\alpha|}_{h},
$$
where the constants $C_{\alpha,\beta,n}$ and $C$ do not depend on $M_2$ and on $h$.

\subsection{Conjugation of $S_{\varepsilon,s}$ by $e^{\lambda_1}(x,D)$}
\

\vskip0.2cm
We recall that 
\begin{align*}
	S_{\varepsilon, s} &:= \langle x \rangle^{s} S_{\varepsilon} \langle x \rangle^{-s} \\
	&=	D_t + \sum_{j=1}^{n} D^{2}_{x_j} + \sum_{j=1}^{n} c_{j,\varepsilon}(t,x) D_{x_j} + 2is \langle x \rangle^{-1} \sum_{j=1}^{n}\frac{x_j}{\langle x \rangle} D_{x_j} + c_{0,\varepsilon, s}(t,x),
\end{align*}
where 
\begin{align*}
	c_{0,\varepsilon, s}(t,x) = c_{0,\varepsilon}(t,x) + is \langle x \rangle^{-1} \sum_{j=1}^{n} c_{1,j}(t,x) \frac{x_j}{\langle x \rangle} - s(s+2) \langle x \rangle^{-2} \frac{|x|^{2}}{\langle x \rangle^{2}} + ns \langle x \rangle^{-2}. 
\end{align*}
Let us set
$$
S_{\varepsilon, s, \lambda_1} = e^{\lambda_1} \circ S_{\varepsilon, s} \circ \{e^{\lambda_1}\}^{-1}.
$$
Applying Theorem \ref{quase_sem_querer} to $\sum\limits_{j=1}^{n} D^{2}_{x_j}$ and formula \eqref{palhas_do_coqueiro} to the first order terms of $S_{\varepsilon, s}$ we get
\begin{align}\label{the_clasnman_3}
	S_{\varepsilon,s,\lambda_1} &= D_t + \sum_{j=1}^{n} D^{2}_{x_j} +2i\sum_{j=1}^{n} (\partial_{x_j}\lambda_1)(x,D) D_{x_j}   \\
	&+\sum_{j=1}^{n} c_{j,\varepsilon}(t,x)D_{x_j}+ 2is \langle x \rangle^{-1} \sum_{j=1}^{n}\frac{x_j}{\langle x \rangle} D_{x_j} \nonumber \\ 
	&+ \textrm{op}\left( \sum_{j=1}^{n} \partial_{\xi_j} \lambda_1 D_{x_j} p_{\varepsilon}\right) + d_{0,\varepsilon,s}(t,x,D), \nonumber
\end{align}
where
\begin{equation}\label{pepsilon}
	p_{\varepsilon}(t,x,\xi) =
	\sum_{j=1}^{n} c_{j,\varepsilon}(t,x)\xi_j + 2is \langle x \rangle^{-1} \sum_{j=1}^{n}\frac{x_j}{\langle x \rangle} \xi_j,
\end{equation}
and $d_{0,\varepsilon, s}$ is a zero order term satisfying
\begin{equation}\label{the_clasnman_5}
	|\partial^{\alpha}_{\xi}\partial^{\beta}_{x} d_{0,\varepsilon,s}(t,x,\xi)| \leq C_{\alpha,\beta,M_1,n,s} \varepsilon^{-J(N,M_1,n)-|\beta|} \langle \xi \rangle^{-|\alpha|}_{h}.
\end{equation}
In order to simplify the expression for $S_{\varepsilon,s,\lambda_1}$ we set
$$
p_{\varepsilon, \log} (t,x,\xi) = \sum_{j=1}^{n} \partial_{\xi_j} \lambda_1(x,\xi) D_{x_j} p_{\varepsilon}(t,x,\xi).
$$
We note that 
$$
|\partial^{\alpha}_{\xi}\partial^{\beta}_{x} p_{\varepsilon, \log} (t,x,\xi)| \leq M_1C_{\alpha,\beta,s,n} \varepsilon^{-N-|\beta|-1} \log \langle \xi \rangle_h \langle \xi \rangle^{-|\alpha|}_h,
$$
and we may write
\begin{align}\label{the_clasnman_4}
	S_{\varepsilon,s,\lambda_1} &= D_t + \sum_{j=1}^{n} D^{2}_{x_j} + \sum_{j=1}^{n} c_{j,\varepsilon}(t,x)D_{x_j} \\
	&+2is \langle x \rangle^{-1} \sum_{j=1}^{n}\frac{x_j}{\langle x \rangle} D_{x_j} + 2i\sum_{j=1}^{n} (\partial_{x_j}\lambda_1)(x,D) D_{x_j} \nonumber \\
	&+ p_{\varepsilon,\log}(t,x,D) + d_{0,\varepsilon,s}(t,x,D). \nonumber
\end{align}

\subsection{Conjugation of $S_{\varepsilon, s, \lambda_1}$ by $e^{\lambda_2}(x,D)$}
\

We shall denote 
$$
S_{\varepsilon, s, \lambda_1, \lambda_2} = e^{\lambda_2}(x,D) \circ S_{\varepsilon, s,\lambda_1} \circ \{e^{\lambda_2}(x,D)\}^{-1}.
$$
We recall that $e^{\lambda_2}$ has order zero. So, using  Lemma \ref{die_young_sabbath}, Theorem \ref{Hotel_california} and repeating similar ideas of the previous subsection we end up with
\begin{align}\label{the_clasnman}
	S_{\varepsilon,s,\lambda_1, \lambda_2} &= D_t + \sum_{j=1}^{n} D^{2}_{x_j} + \sum_{j=1}^{n} c_{j,\varepsilon}(t,x)D_{x_j} + 2i\sum_{j=1}^{n} (\partial_{x_j}\lambda_2)(x,D) D_{x_j} \\
	&+ 2is \langle x \rangle^{-1} \sum_{j=1}^{n}\frac{x_j}{\langle x \rangle} D_{x_j} + 2i\sum_{j=1}^{n} (\partial_{x_j}\lambda_1)(x,D) D_{x_j} \nonumber \\
	&+ p_{\varepsilon,\log}(t,x,D) + d_{0,\varepsilon,s}(t,x,D), \nonumber
\end{align}
for a new zero order term $d_{0,\varepsilon,s}(t,x,\xi)$ satisfying
\begin{equation}\label{the_clasnman_2}
	|\partial^{\alpha}_{\xi}\partial^{\beta}_{x} d_{0,\varepsilon,s}(t,x,\xi)| \leq C_{\alpha,\beta,M_1,n,s} \varepsilon^{-J(N,M_1,n) - |\beta|} e^{CM_2}\langle \xi \rangle^{-|\alpha|}_{h},
\end{equation}
where $J(N, M_1, n)$ is a natural number depending on $N, M_1$ and on the dimension and the constants $C_{\alpha,\beta,M_1,n,s}$, $C$ do not depend on $\varepsilon, M_2$ and $h$.

\subsection{Choices of $M_1$, $M_2$, $h$ and $L^2$ energy estimate}

Inserting \eqref{pepsilon} in  $p_{\varepsilon,\log}(t,x,\xi)$ we obtain
\begin{align*}
	p_{\varepsilon, \log}(t,x,\xi) &= \sum_{j=1}^{n} \sum_{\ell = 1}^{n} \partial_{\xi_j} \lambda_1(x,\xi) D_{x_j} c_{\ell,\varepsilon}(t,x) \xi_\ell 
	+ 2i \sum_{j=1}^{n} \sum_{\ell = 1}^{n} \partial_{\xi_j} \lambda_1(x,\xi) D_{x_{j}} \frac{x_\ell}{\langle x \rangle^{2}} \xi_\ell \\
	&= q_{1,\varepsilon}(t,x,\xi) + q_{2,\varepsilon}(t,x,\xi).
\end{align*}
Since $c_{j,\varepsilon}(t,x)$ are Schwartz functions in $x$ we obtain 
\begin{align*}
	|\partial^{\alpha}_{\xi} \partial^{\beta}_{x} q_{1,\varepsilon}(t,x,\xi)| &\leq M_1C_{\alpha, \beta, s, n} \varepsilon^{-N - |\beta| - 3} \log\langle \xi \rangle_h \langle \xi \rangle^{-|\alpha|}_{h} \langle x \rangle^{-2} \\
	&= M_1C_{\alpha, \beta, s, n } \varepsilon^{-N - |\beta| - 3} \langle \xi \rangle^{-1}_{h}\log\langle \xi \rangle_h \langle \xi \rangle^{1-|\alpha|}_{h}\langle x \rangle^{-2} \\
	&\leq M_1C_{\alpha, \beta, s, n} \varepsilon^{-N - |\beta| - 3} h^{-\frac{1}{2}} \langle \xi \rangle^{1-|\alpha|}_{h} \langle x \rangle^{-2}.
\end{align*}
Analogously
\begin{align*}
	|\partial^{\alpha}_{\xi} \partial^{\beta}_{x} q_{2,\varepsilon}(t,x,\xi)| \leq M_1C_{\alpha, \beta, s, n}  h^{-\frac{1}{2}} \langle \xi \rangle^{1-|\alpha|}_{h} \langle x \rangle^{-2}.
\end{align*}
So, we can consider $q_{1,\varepsilon}$ and $q_{2,\varepsilon}$ as symbols of order $1$ with decay $\langle x \rangle^{-2}$ and seminorms bounded by $M_1C_{\alpha, \beta, s, n}\varepsilon^{-N - |\beta| - 3}h^{-\frac{1}{2}}$.

	In order to derive an $L^2$ energy estimate for $S_{\varepsilon,\lambda}$ we write 
\begin{align}\label{seventh_son_of_a_seventh_son}
	iS_{\varepsilon,s,\lambda_1, \lambda_2} &= \partial_t + i\sum_{j=1}^{n} D^{2}_{x_j} \\ 
	&+ \sum_{j=1}^{n} ic_{j,\varepsilon}(t,x)D_{x_j} - 2\sum_{j=1}^{n} (\partial_{x_j}\lambda_2)(x,D) D_{x_j} \nonumber \\
	&-2s \langle x \rangle^{-1} \sum_{j=1}^{n}\frac{x_j}{\langle x \rangle} D_{x_j} - 2\sum_{j=1}^{n} (\partial_{x_j}\lambda_1)(x,D) D_{x_j} \nonumber \\
	&+iq_{1,\varepsilon}(t,x,D) + iq_{2,\varepsilon}(t,x,D) + id_{0,\varepsilon,s}(t,x,D). \nonumber
\end{align}
We also recall that 
\begin{align*}
		-2 \sum_{j=1}^{n} \xi_j \partial_{x_j} \lambda_{\ell}(x,\xi) \geq M_{\ell} \langle x \rangle^{-\ell} \chi\left( \frac{\langle x \rangle}{|\xi|} \right) (1-\chi)\left( \frac{ |\xi| }{ h } \right) |\xi|, \quad \ell = 1, 2.
\end{align*}
Note that if a symbol $c(x,\xi)$ of order $1$ decays like $\langle x \rangle^{-2}$, then we can split it as follows
\begin{align*}
	c(x,\xi) &= c(x,\xi)\chi\left( \frac{\langle x \rangle}{|\xi|} \right) (1-\chi)\left( \frac{ |\xi| }{ h } \right) \\
	&+ \underbrace{c(x,\xi) (1-\chi)\left( \frac{\langle x \rangle}{|\xi|} \right) (\chi + (1-\chi))\left( \frac{ |\xi| }{ h } \right) + c(x,\xi) \chi\left( \frac{\langle x \rangle}{|\xi|} \right) \chi\left( \frac{ |\xi| }{ h } \right)}_{\text{order zero}}. \\
\end{align*}
So, outside of the support of $\chi\left( \frac{\langle x \rangle}{|\xi|} \right) (1-\chi)\left( \frac{ |\xi| }{ h } \right)$ we have that
$$
iS_{\varepsilon,s,\lambda_1, \lambda_2} = \partial_t + i\sum_{j=1}^{n} D^{2}_{x_j} + \tilde{q}_{0}(t,x,D),
$$
where 
$$
	|\partial^{\alpha}_{\xi}\partial^{\beta}_{x} \tilde{q}_{0}(t,x,\xi)| \leq C_{\alpha,\beta,M_1,n,s} \varepsilon^{-J(N,M_1,n) - |\beta|} e^{CM_2}\langle \xi \rangle^{-|\alpha|}_{h}.
$$
On the other hand, on the support of $\chi\left( \frac{\langle x \rangle}{|\xi|} \right) (1-\chi)\left( \frac{ |\xi| }{ h } \right)$ we can take $M_1(s)$ and $M_2(\varepsilon)$ large in order to get that the symbols of the first order part are non-negative. Indeed, we have
\begin{align*}
	-\chi\left( \frac{\langle x \rangle}{|\xi|} \right) &(1-\chi)\left( \frac{ |\xi| }{ h } \right)\sum_{j=1}^{n} \Im\,c_{j,\varepsilon}(t,x)\xi_j - 2\sum_{j=1}^{n} (\partial_{x_j}\lambda_2)(x,\xi) \xi_j \nonumber \\
	&\geq  \left( M_2 - C(c_j)\varepsilon^{-N-2} \right) |\xi| \langle x \rangle^{-2} \chi\left( \frac{\langle x \rangle}{|\xi|} \right) (1-\chi)\left( \frac{ |\xi| }{ h } \right)
\end{align*}
and
\begin{align*}
	-\chi\left( \frac{\langle x \rangle}{|\xi|} \right) &(1-\chi)\left( \frac{ |\xi| }{ h } \right)2s \langle x \rangle^{-1} \sum_{j=1}^{n}\frac{x_j}{\langle x \rangle} \xi_j - 2\sum_{j=1}^{n} (\partial_{x_j}\lambda_1)(x,\xi) \xi_j \nonumber \\
	&+ \chi\left( \frac{\langle x \rangle}{|\xi|} \right) (1-\chi)\left( \frac{ |\xi| }{ h } \right) (-\Im\, q_{1,\varepsilon}(t,x,\xi) - \Im\,q_{2,\varepsilon}(t,x,D)) \\
	&\geq \left(  M_1 - C(s) - h^{-\frac{1}{2}}C_{s,n,M_1} \varepsilon^{-N-3} \right) |\xi| \langle x \rangle^{-1} \chi\left( \frac{\langle x \rangle}{|\xi|} \right) (1-\chi)\left( \frac{ |\xi| }{ h } \right).
\end{align*}
Hence, choosing 
$$
M_2(\varepsilon) = \frac{C(c_j)\varepsilon^{-N-2}}{2}, \quad M_1(s) = \frac{C(s)}{2}
$$
and taking $h$ large so that
$$
h^{-\frac{1}{2}} C_{s,n,M_1}\varepsilon^{-N-3} < \frac{C(s)}{2},
$$
we obtain that the real part of the first order symbols in \eqref{seventh_son_of_a_seventh_son} are all non-negative.

Hence we may apply sharp G{\aa}rding inequality to get the following energy estimate:
\begin{align}\label{neon_knight}
	\partial_{t} \|u(t)\|^{2}_{L^2} &= 2 Re\, \langle \partial_t u, u \rangle_{L^2} \\
	&\leq \|S_{\varepsilon,s,\lambda_1,\lambda_2} u\|^{2}_{L^2} + \|u\|^{2}_{L^2} + C_{s,n,M_1} \varepsilon^{-J(N, M_1, s)} e^{CM_2} \|u\|^{2}_{L^2} \nonumber.
\end{align}
To obtain \eqref{neon_knight} with general $m-$norms we apply the same ideas to the operator 
$$
\langle D_x \rangle^{m} S_{\varepsilon,s,\lambda_1,\lambda_2} \langle D_x \rangle^{-m}.
$$
which is almost equal to $S_{\varepsilon,s,\lambda_1,\lambda_2}$ except for a zero order term. So, 
\begin{align*}
	\partial_{t} \|u(t)\|^{2}_{H^m} &\leq \|S_{\varepsilon,s,\lambda_1,\lambda_2} u\|^{2}_{H^m} + \|u\|^{2}_{H^m} + C_{m,s,n,M_1} \varepsilon^{-J(N,M_1,s,m)} e^{CM_2} \|u\|^{2}_{H^m}.
\end{align*}
Gronwall's inequality then gives (recalling that $M_1$ depends only on $s$ and $M_2 = C\varepsilon^{-N-2}$)
\begin{align}\label{aces_high}
	\|u(t)\|^{2}_{H^m} &\leq C_{m,s,T,n}\exp \left\{ C_{T,m,n} e^{\varepsilon^{-J(N,s,m)}} \right\} \left( \|u(0)\|^{2}_{H^{m}} + \int_{0}^{t}\|S_{\varepsilon,s,\lambda_1,\lambda_2}(\tau) u(\tau)\|^{2}_{H^m} d\tau \right).
\end{align}

The next proposition is consequence of \eqref{aces_high}.

\begin{Prop}\label{sex_pistols}
	Let $f \in C([0,T]; H^m(\R^{n}))$ and $g \in H^{m}(\R^{n})$. There exists a unique solution $u$ in $C([0,T];H^{m}(\R^{n}))$ to the Cauchy problem 
	\begin{equation}\label{sign_of_the_cross1}
		\begin{cases}
			S_{\varepsilon,s,\lambda_1, \lambda_2} u(t,x) = f(t,x), \quad t \in [0,T], x \in \R^{n}, \\
			u(0,x) = g(x), \qquad \quad \qquad 	x \in \R^{n},
		\end{cases}
	\end{equation}
	and $u$ satisfies \eqref{aces_high}.
\end{Prop}

\begin{proof}[Proof of Proposition \ref{pompem}.]
	We observe that $e^{\lambda_2}(x,D): H^{m} \to H^{m}$ and $e^{\lambda_1}(x,D): H^m \to H^{m-\frac{c(s)}{2}}$, for some $c(s)$ depending on $s$. Hence $e^{\lambda_2}(x,D)$ does not produce any loss of regularity and $e^{\lambda_1}(x,D)$ produces a loss depending on $s$. Then, Proposition \ref{pompem} is a consequence of Proposition \ref{sex_pistols} and the mapping properties of $e^{\lambda_1}(x,D)$ and $e^{\lambda_2}(x,D)$.
\end{proof}

\smallskip


\section{Consistency with the classical theory in the case of regular coefficients}\label{Consistency with regular theory}
\

Suppose that the coefficients of the operator $S$ are $\mathscr{S}(\R^n)$ regular, more precisely:
%
%
$$
c_j \in C([0,T];\mathscr{S}(\R^{n})), \quad j = 0, 1, \ldots, n.
$$
In this situation the Cauchy problem \eqref{CPintro} with data $f \in C([0,T];\mathscr{S}(\R^{n}))$ and $g \in \mathscr{S}(\R^{n})$ admits a unique solution $u \in C([0,T]; \mathscr{S}(\R^{n}))$ satisfying an energy estimate like \eqref{national_acrobat} (with a constant which does not depend on $\varepsilon$). The goal of this section is to verify that the net obtained in Theorem \ref{apprn} converges to $u$ in the $\mathscr{S}(\R^{n})$ topology. 


Let $u_{\varepsilon}$ be the solution of the Cauchy problem \eqref{regularised_CP_true}. Then $u - u_{\varepsilon}$ solves the following Cauchy problem
\begin{equation}\label{reise_reise}
	\begin{cases}
		S\{u-u_{\varepsilon}\}(t,x) = (f-f_{\varepsilon}) + Q_{\varepsilon}u_{\varepsilon}(t,x), \quad t \in [0,T],\, x \in \R^{n}, \\
		\{u-u_{\varepsilon}\}(0,x) = (g-g_{\varepsilon}), \qquad \qquad \qquad \quad  x \in \R^{n},
	\end{cases}
\end{equation}
where 
$$
Q_{\varepsilon} = 
\sum_{j=1}^{n} (c_{j,\varepsilon}(t,x) - c_{j}(t,x))D_{x_j} + (c_{0,\varepsilon}(t,x)-c_0(t,x)).
$$
So, we have the following estimate: for any $m, s \in \N_0$
\begin{align*}\label{nebel}
	\|&\{u-u_{\varepsilon}\}(t)\|^{2}_{H^{m-c(s), s}} \leq  C_{m,s} \left\{ \|g-g_{\varepsilon}\|^{2}_{H^{m,s}} + \int_{0}^{t} \|(f-f_\varepsilon)(\tau) + Q_{\varepsilon} u_{\varepsilon}(\tau)\|^{2}_{H^{m,s}} d\tau  \right\} \\
	&\leq C_{m,s}\|g-g_{\varepsilon}\|^{2}_{H^{m,s}} + C_m \int_{0}^{t} \|(f-f_\varepsilon)(\tau)\|^{2}_{H^{m,s}}  d\tau 
	\\
	&+ C_{m,s} \int_{0}^{t} \sum_{j=1}^{n} \max_{|\alpha| \leq m} \sup_{x \in \R^{n}}|D^{\alpha}_{x}c_{j,\varepsilon}(t,x) - D^{\alpha}_{x}c_{j}(t,x)|^{2} \|u_{\varepsilon}(\tau)\|^{2}_{H^{m+1,s}} d\tau \\
	&+ C_{m,s} \int^{t}_{0} \max_{|\alpha| \leq m}\sup_{x\in\R^n}|D^{\alpha}_{x}c_{0,\varepsilon}(t,x)-D^{\alpha}_{x}c_0(t,x)|^{2} \|u_{\varepsilon}(\tau)\|^{2}_{H^{m,s}} d\tau.
\end{align*}
Now we observe the following:
\begin{itemize}
	\item[(i)] Since in this case we are regularising \textit{regular} functions, it is easy to conclude that the regularisations $c_{j,\varepsilon}$ satisfy uniform estimates with respect to $\varepsilon$. So, in this particular case, we obtain estimate \eqref{national_acrobat} uniformly with respect to $\varepsilon$. Hence, for any $m,s \in \N_0$, 
	$$
	\|u_{\varepsilon}(\tau)\|^{2}_{H^{m,s}} \leq C_{m,s}(f,g), \quad \forall \varepsilon \in (0,\varepsilon_0).
	$$
	
	
	\item[(ii)] We have the following convergence: for any $\beta \in \N^n_0$ and $j = 0, 1, \ldots, n$ it holds 
	$$
	\sup_{t \in [0,T],\, x \in \R^{n}} |\partial^{\beta}_{x}c_{j,\varepsilon}(t,x) - \partial^{\beta}_{x}c_{j}(t,x)| \to 0, \quad \text{as} \, \varepsilon \to 0.
	$$
	
	\item[(iii)] 
	The regularised data satisfy
		$$
		\|g-g_\varepsilon\|_{H^{m,s}} \to 0, \quad \sup_{t \in [0,T]} \|f(t)-f_{\varepsilon}(t)\|_{H^{m,s}} \to 0, \quad \text{when}\,\, \varepsilon \to 0,
		$$
		for all $m, s \in \N_0$.
\end{itemize}


We therefore conclude $u_{\varepsilon} \to u$ in $C([0,T];\mathscr{S}(\R^{n}))$. Summing up, we have proven the following result.

\begin{Prop}
	Assume that $c_j \in C([0,T]; \mathscr{S}(\R^{n}))$ for $j = 0, 1, \ldots, n$. Assume moreover $g \in \mathscr{S}(\R^n)$ and $f \in C([0,T];\mathscr{S}(\R^n))$. Then, 
	any $\mathscr{S}$-very weak solution converges to the unique classical solution in the space $C([0,T];\mathscr{S}(\R^{n}))$. In particular, in the classical case, the limit of very weak solutions always exists and does not depend on the regularisation. 
\end{Prop}

\smallskip


\section{Classical Schwartz well-posedness result}\label{section_Schwartz_classical_result}
\

As we mentioned in the introduction, from the theory developed in this paper we obtain a result concerning the well-posedness in Schwartz spaces of the Cauchy problem \eqref{CPintro} which is new and interesting per se. In this section, we shall state it and explain how to prove it from the ideas used throughout the manuscript.   

Let $S$ be the Schr\"odinger type operator given in \eqref{city_of_tears} with the following hypotheses on the coefficients:
\begin{equation}\label{peace_of_mind}
	c_j \in C([0,T]; \mathcal{B}^{\infty}(\R^n)), \quad |\Im\, c_j(t,x)| \leq C \langle x \rangle^{-1}, \quad j = 0, 1, \ldots, n. 
\end{equation}
Then the following result holds.

\begin{Th}\label{Th_classical_Schwartz_result}
	Suppose $S$ with coefficients satisfying \eqref{peace_of_mind}. Let $f \in C([0,T]; \mathscr{S}(\R^n))$ and $g \in \mathscr{S}(\R^n)$ and consider the Cauchy problem \eqref{CPintro}. Then there exists a unique solution $u \in C([0,T]; \mathscr{S}(\R^n))$. Moreover, the solution $u$ satisfies the following energy inequality   
	\begin{equation*}
		\| u(t) \|^{2}_{H^{m-c(s), s}} \leq  C_{m,s,n,T} \left\{ \|g\|^{2}_{H^{m,s}} + \int_{0}^{t} \|f(\tau)\|^{2}_{H^{m,s}} d\tau \right\}, \quad s,m \in \R, \, t \in [0,T],
	\end{equation*}
	for some constant $c(s)$ depending only on $s$. In particular, the Cauchy problem \eqref{CPintro} is well-posed in $\mathscr{S}(\R^n)$.
\end{Th}

To prove the above theorem, the idea is to consider the operator 
$$
S_s = \langle x \rangle^{s} \circ S \circ \langle x \rangle^{-s}, 
$$
and to repeat the steps of Section \ref{Iron_maiden}. We just remark that since we do not have to worry with a parameter $\varepsilon$ and the coefficients have rapid decay, the sole conjugation by $e^{\lambda_1}(x,D)$ is enough to prove the above result via an auxiliary Cauchy problem (cf. again Section \ref{Iron_maiden}). 

\smallskip


\section*{Appendix}
\addcontentsline{toc}{section}{Appendix}
\renewcommand{\thesection}{A}

	\section{Regularisation of tempered distributions by Schwartz functions}\label{Regularisation of tempered distributions by Schwartz functions}

\noindent

This appendix is devoted to the study of a special kind of regularisations of tempered distributions. Namely, we want to regularise distributions in $\mathscr{S}'(\R^n)$ by functions in $\mathscr{S}(\R^n)$. Despite the results reported here are quite natural and straightforward, we decided to present them in detail for sake of completeness.

We recall that if $u \in \mathscr{S}'(\R^n)$ then there exist $C > 0$ and $N \in \N_0$ such that 
$$
|\langle u , h \rangle| \leq C \sum_{|\alpha|, |\beta| \leq N} \sup_{x\in\R^{n}}\langle x \rangle^{|\beta|}|\partial^{\alpha}_{x} h(x)|, \quad \forall\, h \in \mathscr{S}(\R^{n}).
$$
Let $\phi, \psi \in \mathscr{S}(\R^{n})$ such that $\phi(0) = 1$ and $\widehat{\psi}(0) = 1$. We shall use the following notations
$$
\phi^{\varepsilon}(x) := \phi(\varepsilon x), \quad \psi_{\varepsilon}(\xi) := \frac{1}{\varepsilon^{n}} \psi\left( \frac{\xi}{\varepsilon} \right), \quad \varepsilon \in (0,1],\, x, \xi \in \R^{n}.
$$
For any given $u \in \mathscr{S}'(\R^n)$ we consider the following type of regularisation 
\begin{equation}\label{regularisation}
	u_{\varepsilon}(x) = \phi^{\varepsilon}(x) \{\psi_{\varepsilon}\ast u\}(x).
\end{equation}

\begin{Lemma}
	Let $u \in \mathscr{S}'(\R^{n})$. Then 
	$$
	\psi_{\varepsilon} \ast u \in \mathcal{O}_{M}(\R^{n}),
	$$ where
$$	\mathcal{O}_{M} = \left\{ f \in C^{\infty}(\R^{n}) : \text{for any}\,\, \alpha \in \N_0^n \,\, \text{there is}\, p \geq 0\,\, \text{such that} \,\, \sup_{x} \{\langle x \rangle^{-p} |\partial^{\alpha}_{x}f(x)|\} < \infty \right\}.
	$$
	
	More precisely, the following estimate holds:
	$$
	|\partial^{\alpha}_{x} (\psi_{\varepsilon} \ast u)(x)| \leq  C_{N,\alpha} \varepsilon^{-N-|\alpha|-n} \langle x \rangle^{N},
	$$
	where $N$ depends on $u$ and $C_{N,\alpha}$ depends on $u$, $\alpha$ and on seminorms of $\psi$ with order up to $N+|\alpha|$. 
\\
\end{Lemma}
\begin{proof}
	The proof follows from the following estimate
	\begin{align*}
		|\partial^{\alpha}_{x} (\psi_{\varepsilon} \ast u)(x)| &=  \frac{1}{\varepsilon^{|\alpha|+n}} \left|\left\langle u,  (\partial^{\alpha}\psi)\left(\frac{x-\cdot}{\varepsilon}\right) \right\rangle \right| \\
		&\leq \frac{C}{\varepsilon^{|\alpha|+n}} \sum_{|\gamma|, |\beta| \leq N} \left\| \langle y \rangle^{|\beta|} \partial^{\gamma}_{y} (\partial^{\alpha}\psi)\left(\frac{x-y}{\varepsilon}\right)  \right\|_{L^{\infty}(\R^{n}_{y})} \\
		&\leq \frac{C}{\varepsilon^{|\alpha|+n}} \sum_{|\gamma|, |\beta| \leq N} \left\| \varepsilon^{-|\gamma|} 2^{|\beta|} \langle x \rangle^{|\beta|} \langle x-y\rangle^{|\beta|} (\partial^{\alpha+\gamma}\psi)\left(\frac{x-y}{\varepsilon}\right)  \right\|_{L^{\infty}(\R^{n}_{y})} \\
		&\leq C_{N,\alpha} \varepsilon^{-|\alpha|-N-n} \langle x \rangle^{N}.
	\end{align*}
\end{proof}

From Leibniz formula and the above lemma we easily conclude that $u_{\varepsilon}$ given in \eqref{regularisation} belongs to $\mathscr{S}(\R^{n})$ for all $\varepsilon \in (0,1)$ and the following estimate holds:
\begin{equation}\label{estimate_that_we_need}
	|\langle x \rangle^{M} \partial^{\alpha}_{x} u_{\varepsilon}(x)| \leq C_{N,M,\alpha} \varepsilon^{-n-|\alpha|-N-M}.
\end{equation}

Now we want to check that the regularisation $u_\varepsilon$ converges to $u$ in $\mathscr{S}'(\R^{n})$. This is done in Theorem \ref{drunken_lullabies} here below; to prove it we shall need the following lemmas. 

\begin{Lemma}
	Let $\{u_j\}_{j \in \N_0} \subset \mathscr{S}'(\R^{n})$, $u \in \mathscr{S}'(\R^{n})$, $\{h_j\}_{j \in \N_0} \subset \mathscr{S}(\R^{n})$ and $h \in \mathscr{S}(\R^{n})$. If 
	$$
	\lim_{j \to \infty} u_j = u \,\, \text{in}\,\, \mathscr{S}'(\R^{n}) \quad \text{and} \quad \lim_{j\to\infty}h_j = h \,\, \text{in}\,\, \mathscr{S}(\R^{n}),
	$$ 
	then $u_j(h_j) \to u(h)$ in $\C$.
\end{Lemma}
\begin{proof}
	Since $u_j \to u$ in $\mathscr{S}'(\R^{n})$ we have that the family $\{u_j\}$ is pointwise bounded. Therefore, Banach-Steinhaus theorem gives that the sequence $\{u_j\}$ is equicontinuous. In this way, for any $\delta_1 > 0$ there exists $\delta_2 > 0$ such that 
	$$
	d(h, w) \leq \delta_1,\,\, h,w \in \mathscr{S}(\R^{n}) \implies |u_j(h-w)| \leq \delta_2, \quad \forall\, j \in \N_0,  
	$$
	where $d$ denotes the usual distance in $\mathscr{S}(\R^{n})$. In particular, we conclude that
	$$
	\lim_{j \to \infty} u_j(h_j-h) = 0.
	$$
	To finish the proof it suffices to observe that
	$$
	u_j(h_j) - u(h) = u_j(h_j - h) + u_j(h) - u(h) \to_{j\to\infty} 0.
	$$
\end{proof}

\begin{Lemma}
	For any $h \in \mathscr{S}(\R^{n})$ we have
	$$
	\psi_{\varepsilon} \ast h \to h \,\, \text{in}\; \mathscr{S}(\R^{n}).
	$$
\end{Lemma}
\begin{proof}
	Since $\partial^{\alpha}_{x} \{ \psi_{\varepsilon} \ast h \} = \psi_{\varepsilon} \ast \{\partial^{\alpha}_{x} h\}$ we only need to prove that 
	$$
	\sup_{x \in \R^{n}} \langle x \rangle^{M} |\psi_{\varepsilon} \ast h (x) - h(x)| \to 0, \,\, \text{as} \, \varepsilon \to 0. 
	$$
	We have 
	\begin{align*}
		\psi_{\varepsilon} \ast h (x) - h(x) = \int \psi_{\varepsilon}(y)\{ h(x-y) - h(x)\} dx = \int \psi_{\varepsilon}(y)\, y \cdot \int_{0}^{1} (\nabla h)(x-\theta y) d\theta \, dy.
	\end{align*}
	
	Let $\delta > 0$. When $|y| \leq \delta$ we have 
	\begin{eqnarray*}
		\langle x \rangle^{M} \left| \int_{|y| \leq \delta} \psi_{\varepsilon}(y)\, y \right.\hskip-0.4cm&\cdot&\hskip-0.4cm\left. \int_{0}^{1} (\nabla h)(x-\theta y) d\theta \, dy \right| \\
&\leq & \int_{|y| \leq \delta} \psi_{\varepsilon}(y) |y| \int_{0}^{1} 2^{M} |(\nabla h)(x-\theta y)| \langle x - \theta y \rangle^{M} \langle \theta y \rangle^{M} d\theta \, dy  \\
		&\leq& \delta 4^{M} \sup_{x \in \R^n} |(\nabla h)(x)| \langle x \rangle^{M}.
	\end{eqnarray*}
	Since 
	\begin{eqnarray*}
	\langle x \rangle^{M} |h(x-y) - h(x)| &\leq& 2^{M} \langle y \rangle^{M} \langle x - y \rangle^{M} |h(x-y)| + \langle x \rangle^{M} |h(x)| 
\\
&\leq& 2^{M+1} \langle y \rangle^{M} \sup_{x \in \R^{n}} \langle x \rangle^{M} |h(x)|,
	\end{eqnarray*}
	if $|y| \geq \delta$ we get
	\begin{align*}
		\langle x \rangle^{M} \left| \int_{|y| \geq \delta} \psi_{\varepsilon}(y)\{ h(x-y) - h(x)\} dx \right| &\leq 2^{M+1} \sup_{x \in \R^{n}} \langle x \rangle^{M} |h(x)| \int_{|y| \geq \delta}  \psi_{\varepsilon}(y) \langle y \rangle^{M} dy \\
		&\leq  2^{M+1} \sup_{x \in \R^{n}} \langle x \rangle^{M} |h(x)| \underbrace{\int_{|y| \geq \frac{\delta}{\varepsilon}}  \psi(y) \langle y \rangle^{M} dy}_{\to 0 \,\, \text{as} \, \varepsilon \to 0}.
	\end{align*}
	
	From the above inequalities we can find $\varepsilon_0 > 0$ such that 
	$$
	\varepsilon \leq \varepsilon_0 \implies \sup_{x \in \R^{n}} \langle x \rangle^{M} |\psi_{\varepsilon} \ast h (x) - h(x)|  \leq C\delta.
	$$
\end{proof}

\begin{Rem}
	From the previous lemma we conclude that for any $u \in \mathscr{S}'(\R^{n})$ 
	$$
	\psi_{\varepsilon} \ast u \to u \,\, \text{in} \,\, \mathscr{S}'(\R^{n}).
	$$ 
	Indeed, it suffices to observe
	$$
	\langle \psi \ast u, h \rangle = \langle u, \tilde{\psi} \ast h \rangle, \quad h \in \mathscr{S}(\R^{n}),
	$$
	where $\tilde{\psi}(\xi) = \psi(-\xi)$, for any $\psi \in \mathscr{S}(\R^n)$ and $u \in \mathscr{S}'(\R^n)$.
\end{Rem}

\begin{Lemma}
	For any $h \in \mathscr{S}(\R^{n})$ we have
	$$
	\phi^\varepsilon h \to h \,\, \text{in}\, \mathscr{S}(\R^{n}).
	$$
\end{Lemma}
\begin{proof}
	We have 
	\begin{align*}
		\langle x \rangle^{M}\partial^{\alpha}_{x}\{(\phi^\varepsilon(x)-1)h(x)\} &= \langle x \rangle^{M}\{\phi(\varepsilon x) - 1\}\partial^{\alpha}_{x} h(x) \\ 
		&+ 
		\underbrace{\varepsilon \langle x \rangle^{M} \sum_{\overset{\alpha_1 + \alpha_2 = \alpha}{\alpha_1 \geq 1}} \frac{\alpha!}{\alpha_1!\alpha_2!} \varepsilon^{|\alpha_1|-1} (\partial^{\alpha_1}\phi)(\varepsilon x) \partial^{\alpha_2}_{x}h(x)}_{\text{converges to}\, 0\, \text{uniformly in}\, x }.
	\end{align*}
	So we only need to prove that $ \langle x \rangle^{M}\{\phi(\varepsilon x) - 1\}\partial^{\alpha}_{x} h(x)$ converges to zero uniformly in $x$. We first write 
	\begin{align*}
		\langle x \rangle^{M}\{\phi(\varepsilon x) - 1\}\partial^{\alpha}_{x} h(x) = \langle x \rangle^{-1} \{\phi(\varepsilon x) - 1\} \langle x \rangle^{M+1} \partial^{\alpha}_{x}h(x).
	\end{align*}
	Let then $\delta > 0$. There is $R_{\delta} > 0$ such that 
	\begin{align*}
		|x| \geq R_\delta &\implies \langle x \rangle^{-1} \leq \delta \left\{ 2 \|\phi\|_{L^{\infty}} \|\langle x \rangle^{M+1} |\partial^{\alpha}_{x}h(x)| \|_{L^{\infty}} \right\}^{-1}  \\
		&\implies  \langle x \rangle^{-1} \{\phi(\varepsilon x) - 1\} \langle x \rangle^{M+1} |\partial^{\alpha}_{x}h(x)| \leq \delta.
	\end{align*}
	On the other hand, since $\phi_{\varepsilon}(x)$ converges uniformly to $1$ on compact sets, we may find $\varepsilon_0 > 0$ such that 
	\begin{align*}
		|x| \leq R_{\delta}, \,\, \varepsilon \leq \varepsilon_0 \implies  |\phi(\varepsilon x) - 1|\langle x \rangle^{M} |\partial^{\alpha}_{x} h(x)| \leq \delta.
	\end{align*}
	Hence, if $\varepsilon \leq \varepsilon_0$ we have
	$$
	\sup_{x \in \R^n} |\phi(\varepsilon x) - 1| \langle x \rangle^{M} |\partial^{\alpha}_{x} h(x)| \leq \delta.
	$$
\end{proof}

\begin{Rem}
	From the previous two lemmas we conclude $u_\varepsilon \to u$ in $\mathscr{S}(\R^{n})$ provided that $u \in \mathscr{S}(\R^{n})$.
\end{Rem}

We are finally ready to prove that the regularisation $u_{\varepsilon}$ defined by \eqref{regularisation} converges to $u$ in $\mathscr{S}'(\R^{n})$. Indeed, this follows immediately from the previous three lemmas. As a matter of fact
\begin{align*}
	\langle u_{\varepsilon}, h \rangle = \langle \underbrace{\psi_{\varepsilon} u}_{\to u \,\, in \, \mathscr{S}'} , \underbrace{\phi_{\varepsilon} h}_{\to h\, in \, \mathscr{S}(\R^{n})} \rangle \to \langle u , h \rangle, \quad h \in \mathscr{S}(\R^{n}).	
\end{align*}

We summarise what we have proven so far in this Appendix in the following result.

\begin{Th}\label{drunken_lullabies}
	Let $\phi, \psi \in \mathscr{S}(\R^{n})$ such that $\phi(0) = 1$ and $\widehat{\psi}(0) = 1$ and define
	$$
	\phi^{\varepsilon}(x) := \phi(\varepsilon x), \quad \psi_{\varepsilon}(\xi) := \frac{1}{\varepsilon^{n}} \psi\left( \frac{\xi}{\varepsilon} \right), \quad \varepsilon \in (0,1],\, x, \xi \in \R^{n}.
	$$
	For any given $u \in \mathscr{S}'(\R^n)$ consider
	\begin{equation*}
		u_{\varepsilon}(x) = \phi^{\varepsilon}(x) \{\psi_{\varepsilon}\ast u\}(x).
	\end{equation*}
	Then $u_{\varepsilon} \in \mathscr{S}(\R^{n})$ for any $\varepsilon$, $u_{\varepsilon} \to u$ in $\mathscr{S}'(\R^{n})$ and 
	\begin{equation*}
		|\langle x \rangle^{M} \partial^{\alpha}_{x} u_{\varepsilon}(x)| \leq C_{N,M,\alpha} \varepsilon^{-n-|\alpha|-N-M}.
	\end{equation*}
\end{Th}

\begin{Rem}
	Of course the above theorem still holds if we replace $\varepsilon$ by another positive scale $\omega(\varepsilon)$.
\end{Rem}

It is also of our interest to regularise curves of temperate distributions. 

\begin{Prop}
	Let $u \in C([0,T]; \mathscr{S}'(\R^{n}))$ and set
	$$
	u_{\varepsilon}(t, x) = \phi(\varepsilon x) \{\psi_{\varepsilon} \ast u(t)\}(x), \quad t \in [0,T], x\in \R^{n}.
	$$
	Then $u_{\varepsilon} \in C([0,T];\mathscr{S}(\R^{n}))$ and the following estimate holds
	\begin{equation}\label{iron_maiden}
		|\langle x \rangle^{M} \partial^{\beta}_{x} u_{\varepsilon}(t,x)| \leq C \varepsilon^{-N-M-|\beta|-n},
	\end{equation}
	where $N > 0$ and $C > 0$ do not depend on $t$. 
\end{Prop}
\begin{proof}
	Banach-Steinhaus theorem implies that $\{u(t)\}_{t \in [0,T]}$ is equicontinuous. Hence, there exists an open neighborhood $V \subset \mathscr{S}(\R^{n})$ of $0$ such that 
	$$
	u(t)(V) \subset B(0,1), \quad \forall\, t \in [0,T],
	$$ 
	where $B(0,1) = \{z \in \C: |z| < 1\}$. Since $V$ is open, from the topology of $\mathscr{S}(\R^{n})$, we find $N > 0$ and $r > 0$ such that
	$$
	\{h \in \mathscr{S}(\R^{n}): \rho_N(h) < r\} \subset V,
	$$
	where 
	$$
	\rho_N(h) = \sum_{|\beta|, M \leq N} \sup_{x \in \R^n} \langle x \rangle^{M}|\partial^{\beta}_{x}h(x)|.
	$$
	Hence for any $t \in [0,T]$ and any $h \in \mathscr{S}(\R^{n})$ we get
	$$
	|u(t)(h)| = \frac{2\rho_N(h)}{r} \left| u(t) \left( \frac{rh}{2\rho_N(h)} \right) \right| \leq \frac{2}{r} \rho_N(h).
	$$
	This means that we can find $C, N > 0$ independent of $t \in [0,T]$ such that  
	\begin{equation}\label{i_want_conquer_the_world}
		|u(t)(h)| \leq C \sum_{|\beta|, M \leq N} \sup_{x \in \R^{n}} \langle x \rangle^{M}|\partial^{\beta}_{x}h(x)|, \quad h \in \mathscr{S}(\R^{n}).
	\end{equation}
	From \eqref{i_want_conquer_the_world} we immediately get \eqref{iron_maiden}. 
	
	Now let us prove that $u_{\varepsilon}$ is continuous in $t$. Let $t_{k} \to t_0$ in $[0,T]$. It suffices to prove that (for any fixed $\varepsilon$)
	$$
	\sup_{x\in\R^n} \langle x \rangle^{M} |\phi(\varepsilon x) (\{u(t_k) - u(t_0)\} \ast \psi_\varepsilon)(x)| \to 0, \,\, \text{as} \,\, k \to \infty.
	$$ 
	Denote by $f_k (x) = \langle x \rangle^{M} \phi(\varepsilon x) (\{u(t_k)-u(t_0)\} \ast \psi_\varepsilon)(x)$. Since $\phi \in \mathscr{S}(\R^n)$, and from \eqref{i_want_conquer_the_world} we have 
	$$
	|\{u(t_k)-u(t_0)\} \ast \psi_{\varepsilon} (x)| \leq C \varepsilon^{-N-n} \langle x \rangle^{N}, 
	$$
	we conclude that 
	$$
	\sup_{k,x} |f_k(x)| < + \infty. 
	$$On the other hand, the equicontinuity of $\{u(t)\}_{t}$ implies equicontinuity of $\{f_k\}_{k}$. Since $f_k \to 0$ pointwise, applying Ascoli-Arzel\`a theorem we get that (passing to a subsequence if necessary) $f_k \to 0$ uniformly in $\R^{n}$. We can conclude that $u_{\varepsilon}(t_k) \to u(t_0)$ in $\mathscr{S}(\R^{n})$ for any fixed $\varepsilon$.
\end{proof}

\smallskip


\addcontentsline{toc}{section}{Appendix}
\renewcommand{\thesection}{B}
\section{Pseudodifferential operators}\label{pseudodifferential_operators}
\

This appendix is devoted to report some results and definitions concerning pseudodifferential operators that were employed throughout the paper. For the proofs we address the reader to \cite{Kumano-Go}.

\begin{Def}
	For a given $m\in\R$, $S^{m}(\R^{2n})$ denotes the space of all smooth functions $p(x,\xi) \in C^{\infty}(\R^{2n})$ such that for any $\alpha, \beta \in \N^{n}_{0}$ there exists a positive constant $C_{\alpha,\beta}$ for which
	$$
	|\partial^{\alpha}_{\xi} \partial^{\beta}_{x} p(x,\xi)| \leq C_{\alpha,\beta} \langle \xi \rangle^{m-|\alpha|}.
	$$
	The Fr\'echet topology of the space $S^{m}(\R^{2n})$ is induced by the following family of seminorms
	$$
	|p|^{(m)}_\ell := \max_{|\alpha| \leq \ell, |\beta| \leq \ell} \sup_{x, \xi \in \R^{2n}} |\partial^{\alpha}_{\xi} \partial^{\beta}_{x} p(x,\xi)| \langle \xi \rangle^{-m + |\alpha|},\quad\ell \in \N_0.
	$$
\end{Def}


Given a symbol $p(x,\xi)$  we associate to it the operator
$$
p(x,D) u(x) = \int e^{i\xi x} p(x,\xi) \widehat{u}(\xi) \dslash\xi, \quad u \in \mathscr{S}(\R^{n}),
$$
which maps continuously  $\mathscr{S}(\R^{n})$ into itself.

We will sometimes also use the notation $p(x,D) = {\rm op}(p(x,\xi))$. The next result concerns the action of such operator on Sobolev spaces.

\begin{Th}\label{theorem_Calderon_Vaillancourt}[Calder\'on-Vaillancourt]
	Let $p \in S^{m}(\R^{2n})$. Then for any real number $s \in \R$ there exist $\ell := \ell(s,m) \in \N_0$ and $C:= C_{s,m} > 0$ such that 
	\begin{equation*}
		\| p(x,D)u \|_{H^{s}} \leq C |p|^{(m)}_{\ell} \| u \|_{H^{s+m}}, \quad \forall \, u \in H^{s+m}.
	\end{equation*}
	Moreover, when $m = s = 0$ we can replace $|p|^{(m)}_{\ell}$ by
	\begin{equation*}
		\max_{|\alpha| \leq \ell_1, |\beta| \leq \ell_2} \sup_{x, \xi \in \R^{n}} |\partial^{\alpha}_{\xi} \partial^{\beta}_{x} p(x,\xi)| \langle \xi \rangle^{|\alpha| },
	\end{equation*}
	where 
	$$
	\ell_1 = 2\left\lfloor\frac{n}{2} +1\right\rfloor, \quad \ell_2 = 2\left\lfloor \frac{n}{2} + 1 \right\rfloor.
	$$
\end{Th}

In the following we consider the algebra properties of $S^{m}(\R^{2n})$ with respect to the composition of operators. In the sequel $Os-$ in front of the integral sign stands for oscillatory integral. Let $p_j \in S^{m_j}(\R^{n})$, $j = 1, 2$, and define 
\begin{align}\label{eq_symbol_of_composition}
	q(x,\xi) &= Os- \iint e^{-iy \eta} p_1(x,\xi+\eta)p_2(x+y,\xi) dy \dslash\eta \\
	&= \lim_{\mu \to 0} \iint e^{-iy\eta} p_1(x,\xi+\eta)p_2(x+y,\xi) e^{-\mu^2|y|^2} e^{-\mu^2|\eta|^2} dy \dslash\eta. \nonumber
\end{align} 

\begin{Th}\label{Hotel_california}
	Let $p_j \in S^{m_j}(\R^{2n})$, $j = 1, 2$, and consider $q$ defined by \eqref{eq_symbol_of_composition}. Then $q \in S^{m_1+m_2}(\R^{2n})$ and $q(x,D) = p_1(x,D) p_2(x,D)$. Moreover, the symbol $q$ has the following asymptotic expansion
	\begin{align*}
		q(x,\xi) = \sum_{|\alpha| < N} \frac{1}{\alpha!} \partial^{\alpha}_{\xi}p_{1}(x,\xi)D^{\alpha}_{x}p_2(x,\xi) + r_N(x,\xi), 
	\end{align*} 
	where 
	$$
	r_N(x,\xi) = \sum_{|\gamma| = N} \frac{N}{\gamma!} \int_{0}^{1} (1-\theta)^{N-1} \, Os - \iint e^{-iy\eta} \partial^{\gamma}_{\xi} p_1(x,\xi+\theta\eta) D^{\gamma}_{x} p_2(x+y,\xi)  dy\dslash\eta \, d\theta,
	$$
	and the seminorms of $r_N$ may be estimated in the following way: for any $\ell_{0} \in \N_0$ there exists $C_{\ell, N, n} > 0$ such that 
	$$
	|r_N|^{(m_1+m_2)}_{\ell_0} \leq C_{\ell,N,n} |p_1|^{(m_1)}_{\ell_0+N+n+1} |p_2|^{(m_2)}_{\ell_0+N+n+1}.
	$$
\end{Th}

The last theorem that we recall is the celebrated sharp G{\aa}rding inequality (see for instance Theorem 2.1.3 in \cite{Kenig_SG}).

\begin{Th}\label{Theorem_sharp_garding}
	Let $p \in S^{1}(\R^{2n})$ and suppose $Re\,p(x,\xi) \geq 0$ for all $x \in \R^n$ and $|\xi| \geq R$ for some $R>0$. Then there exist $k = k(n) \in \N_0$ and $C = C(n,R)$ such that 
	$$
	Re\, \langle p(x,D) u, u \rangle_{L^{2}} \geq -C |p|^{(1)}_{k} \|u\|^{2}_{L^{2}}, \quad u \in \mathscr{S}(\R^{n}).
	$$ 
\end{Th}

\smallskip


\section*{Acknowledgments}
\

The first author was supported by the grant  $2022/01712-3$ from S\~ao Paulo Research Foundation (FAPESP). 
The second and third author were supported by the INdAM-GNAMPA projects CUP E53C22001930001 and CUP E53C23001670001. The fourth author was supported by the
EPSRC grant EP/V005529/2.

\smallskip



\end{document}